\documentclass[sn-mathphys]{JORSC-jnl}

\jyear{2023}%

\theoremstyle{thmstyleone}%
\newtheorem{theorem}{Theorem}[section]
\usepackage{amsthm} 
\newtheorem{example}{Example}%
\newtheorem{remark}{Remark}[section]%
\newtheorem{corollary}{Corollary}[section]%
\newtheorem{lemma}{Lemma}[section]

\theoremstyle{thmstylethree}%
\usepackage{indentfirst}
\usepackage{epstopdf}

\usepackage{booktabs}
\usepackage{pifont} 
\usepackage{adjustbox}
\usepackage{multirow} 

\usepackage{subcaption}
\usepackage{graphicx}
\usepackage{amsmath}
\usepackage{amsthm}
\usepackage{amssymb}
\usepackage{mathtools}
\def\ST{\songti\rm\relax}
\def\R{{\mathbb R}}

\def\ST{{\rm s.t.}}
\def\argmin{{\rm argmin}}

\begin{document}
\title[Subgradient Gliding Method]{\vspace{-12mm} Subgradient Gliding Method for Nonsmooth Convex Optimization}


\author[1]{Zhihan Zhu}\email{zhihanzhu@buaa.edu.cn}

\author[1]{Yanhao Zhang}\email{yanhaozhang@buaa.edu.cn}

\author*[1]{Yong Xia}\email{yxia@buaa.edu.cn}

\affil[1]{\orgdiv{LMIB of the Ministry of Education, School of Mathematical Sciences}, \orgname{Beihang University}, \orgaddress{ \city{Beijing} \postcode{100191}, \country{People's Republic of China}}}

\abstract{\unboldmath We identify and analyze a fundamental limitation of the classical projected subgradient method in nonsmooth convex optimization: the inevitable failure caused by the absence of valid subgradients at boundary points. We show that, under standard step sizes for both convex and strongly convex objectives, the method can fail after a single iteration with probability arbitrarily close to one, even on simple problem instances. To overcome this limitation, we propose a novel alternative termed the \textit{subgradient gliding method}, which remains well defined without boundary subgradients and avoids premature termination. Beyond resolving this foundational issue, the proposed framework encompasses the classical projected subgradient method as a special case and substantially enlarges its admissible step-size design space, providing greater flexibility for algorithmic design. We establish optimal ergodic convergence rates, $\mathcal{O}(1/\sqrt{t})$ for convex problems and $\mathcal{O}(1/t)$ for strongly convex problems, and further extend the framework to stochastic settings. Notably, our analysis does not rely on global Lipschitz continuity of the objective function, requiring only mild control on subgradient growth. Numerical experiments demonstrate that, in scenarios where the classical projected subgradient method fails completely, the proposed method converges reliably with a $100\%$ success rate and achieves orders-of-magnitude improvements in accuracy and convergence speed. These results substantially expand the scope of subgradient-based optimization methods to non-Lipschitz nonsmooth convex problems.
}




\maketitle

\section{Introduction}
To address the nonsmooth convex optimization problem
\begin{equation*}
	x^*\in\argmin_{x\in\mathcal{X}}f(x),
\end{equation*}
where $\mathcal{X}\subset\R^n$ is a compact convex set contained in the Euclidean ball $B(x^*,R)$, and $f$ is a (possibly nonsmooth) convex function,  the classical projected subgradient method (PSG) proceeds iteratively as follows:
\begin{equation*}
	\left\{	\begin{aligned}
		y_{s+1}&=x_s-\alpha_s g_s, ~\text{where } g_s\in\partial f(x_s),\\
		x_{s+1}&=  \argmin_{x\in\mathcal{X}}\Vert x-y_{s+1} \Vert,	\end{aligned} \right.
\end{equation*}
where $\|\cdot\|$ denotes the Euclidean norm throughout this paper.

\subsection{Recent Advances on Projected Subgradient Method}
When the objective function $f$ is Lipschitz continuous with a known Lipschitz constant $L$ (i.e., $\Vert g\Vert \le L$ for any $g\in\partial f(x)\neq\emptyset$ and $x\in \mathcal{X}$), employing a constant step size of
\begin{equation}\label{size1}
	\alpha_s\equiv\frac{R}{L\sqrt{t}}, ~s=1,\cdots,t,\footnote{$t$ represents the current total number of iterations, while $s$ serves as the indicator of the iteration process.}
\end{equation}
enables PSG to achieve the optimal ergodic convergence rate, which is given by 
\[
f\left(\frac{\sum_{s=1}^t x_s}{t}\right)-f(x^*)\leq \frac{RL}{\sqrt{t}},
\]
see, for example, \cite{nesterov2004lectures,nesterov2018lectures,bubeck2015convex}.

While constant step sizes require predefining the number of iterations, recent studies have demonstrated that appropriately designed time-varying step sizes, as presented in \cite{nesterov2004lectures,nesterov2018lectures,bubeck2015convex}, given by
\begin{equation}
	\alpha_s=\frac{R}{L\sqrt{s}},~s=1,\cdots,t,\label{size2}
\end{equation}
has been proven to ensure the optimal convergence rate of  PSG as well \cite{Zhu}:
\begin{equation}
	f\left(\frac{\sum_{s=1}^t x_s}{t}\right)-f(x^*)\leq \frac{3RL}{2\sqrt{t}}. \label{conv}
\end{equation}

In \cite{nesterov2004lectures,nesterov2018lectures,beck2017first}, further analysis is provided for a more practical subgradient-normalized time-varying step-size scheme:
\begin{equation}
	\alpha_s=\frac{R}{\|g_s\|\sqrt{s}},~s=1,\cdots,t, \label{Nesterov}
\end{equation}
which notably eliminates the need for prior knowledge of the Lipschitz constant. Nevertheless, convergence guarantees for PSG still rely on the Lipschitz continuity assumption of $f$. Furthermore, the resulting convergence rate remains sub-optimal, characterized by
\begin{equation}\label{Nessub}
	f\left(\frac{\sum_{s=1}^t \alpha_s x_s}{\sum_{s=1}^t \alpha_s}\right)-f(x^*)\leq \frac{2RL+RL\log t}{4(\sqrt{t+1}-1)}.
\end{equation}

Recently, \cite{Xia} introduced a novel family of time-varying step-sizes
\begin{equation}
	\alpha_s=\frac{R}{G_s s^{\frac{a}{2}}}, ~s=1,\cdots,t,\label{size3}
\end{equation}
for any fixed $a \in [0,1]$ where 
\begin{equation}
	G_s = \max\{G_{s-1}, \|g_s\|s^{\frac{1-a}{2}}\} ~(G_0=-\infty), \label{G_s}
\end{equation}
which operate without knowledge of the Lipschitz constant, while effectively preserving the weak ergodic convergence properties established in \cite{Zhu}. This family of step-sizes attains the optimal ergodic convergence rate as follows:
\begin{equation}\label{conver}
	f\left(\frac{\sum_{s=1}^t x_s}{t}\right)-f(x^*)\leq
	\frac{3R}{2\sqrt{t}}\cdot \max_{s=1,\cdots,t}\|g_s\|,
\end{equation}
which provides a convergence analysis that does not rely on the Lipschitz assumption.

The aforementioned work addresses the case where the objective function is convex. When the objective function $f$ is $\mu$-strongly convex, incorporating strong convexity information into the step size design
\begin{equation}
	\alpha_s=\frac{2}{\mu(s+1)}, ~s=1,\cdots,t,\label{size4}
\end{equation}
can achieve the optimal ergodic convergence rate of $\mathcal{O}(1/t)$, given by \cite{Lacoste-Julien}
\begin{equation}\label{conver1}
	f\left(\sum_{s=1}^t\frac{2s}{t(t+1)}x_s\right)-f(x^*)\leq
	\frac{2L^2}{\mu(t+1)}.
\end{equation}

\subsection{Fundamental Assumption and Limitation in PSG}

The following theorem concerning the existence of subgradients was originally established in \cite[Theorem 23.4]{Rockafellar} and subsequently appeared in \cite[Proposition 1.1]{bubeck2015convex} and \cite[Theorem 3.1.15]{nesterov2018lectures}. Here, we always assume that the convex set $\mathcal{X}$ has a nonempty interior (i.e., $\text{int}(\mathcal{X}) \neq\emptyset$). In fact, if the affine hull of $\mathcal{X}$ has a lower dimension than the ambient space, the theorem remains valid when replacing the interior with the relative interior.

\begin{theorem}\label{thm1}
	If $f$ is convex then for any $x\in\text{int}(\mathcal{X})$, $\partial f(x)\neq\emptyset$.
\end{theorem}

According to Theorem \ref{thm1}, while a convex function on a convex set is assured to have nonempty subgradients at interior points, such existence cannot be guaranteed at boundary points. Three illustrative examples are provided below.

\begin{example}\label{e1}
	The function $f(x)=-\sqrt{r - k_1x(1)^2-k_2x(2)^2}$, where $r, k_1, k_2 > 0$, is a convex function on the ellipsoid $k_1x(1)^2+k_2x(2)^2\le r$, which lacks subgradients at boundary $k_1x(1)^2+k_2x(2)^2= r$.
\end{example}

\begin{example}\label{e2}
	The function inspired by \cite{Renegar}
	$$
	f(x) = 
	\begin{cases}
		0 & \text{if } (x(1), x(2)) = (0,0), \\
		\frac{x(1)^2 + x(2)^2}{x(1)} & \text{if } x(1) > 0, \\
		\infty & \text{otherwise}.
	\end{cases}
	$$
	is convex over the rectangular domain $0 \le x(1) \le 1, -1 \le x(2) \le 1$, but it lacks subgradients along the boundary segment  $ x(1) = 0, -1 \le x(2) \le 1, x(2) \neq 0$.
\end{example}

\begin{example}\label{e3}
	The negative entropy function $f(x)=\sum_{i=1}^{n} x(i) \log x(i)$ on the hypercube $\mathcal{X} = \{x \in\R_+^n: 0 \le x(i) \le B, ~B >0\}$ is (strongly) convex (with the convention that $0\log0=0$), but it does not admit subgradients at boundary points (i.e., $x(i) = 0$ for some $i$).
\end{example}

Since PSG readily projects iterates onto the boundary, the algorithm cannot proceed when subgradients are absent at these boundary points. Consequently, the existing PSG methods implicitly assume the existence of subgradients at boundary points to function properly as in \cite[Assumption 1.1]{zamani2025exact} and \cite{alber1998projected}, or ensure subgradient existence by imposing additional assumptions on the feasible set $\mathcal{X}$ \cite[Assumption 8.7]{beck2017first}. However, these assumptions are often violated for many convex functions over convex sets, even for the simple examples mentioned above. This limitation restricts PSG's applicability to general nonsmooth convex optimization. In Section 2, we present a more detailed analysis demonstrating that these simple examples may even induce worst-case behavior of the classical PSG.

\subsection{Our Contributions: Subgradient Gliding Method}

This paper's main contributions can be summarized along three closely related aspects.

\textit{Limitations of the classical projected subgradient method:} We first identify and analyze fundamental failure modes of the classical PSG that arise when valid subgradients do not exist on the boundary of the feasible set. In Section 2, we show that under standard step sizes for both convex and strongly convex objectives, PSG can fails after a single iteration with probability arbitrarily close to one, even on simple problem instances. These theoretical findings are further corroborated by numerical experiments in Section 6. Together, these results reveal intrinsic limitations of classical PSG that are not captured by existing convergence analyses.

\textit{A novel alternative with theoretical guarantees:} To overcome the above limitations, existing literature typically imposes additional assumptions to guarantee the existence of subgradients, which, however, further restrict the applicability of PSG. In contrast, we propose the subgradient gliding method (SGM), which guarantees all iterates remain strictly interior, thereby well defined even in the absence of boundary subgradients without requiring any additional assumptions. We establish optimal ergodic convergence rates of $\mathcal{O}(1/\sqrt{t})$ for convex problems and $\mathcal{O}(1/t)$ for strongly convex problems (Sections 3 and 4). We further extend the framework to stochastic settings and develop corresponding convergence guarantees (Section 5). Importantly, the convergence results could be obtained without assuming global Lipschitz continuity of the objective function; instead, convergence is ensured under mild growth conditions on the subgradient norms, thereby significantly expanding the scope of subgradient-based optimization to non-Lipschitz regimes.

\textit{Theoretical and practical advantages:} Beyond resolving the failure of classical PSG, the subgradient gliding framework substantially enlarges the design space of subgradient-based optimization methods. The admissible stepsize rules strictly contain those of PSG, while allowing significantly greater flexibility in practice (Sections 3–5). Numerical experiments in Section 6 demonstrate that, in scenarios where classical PSG fails, SGM achieves stable convergence with a $100\%$ success rate and exhibits orders-of-magnitude improvements in both accuracy and convergence speed. These results highlight that SGM not only enjoys strong theoretical guarantees but also offers practical advantages in stepsize selection and algorithm design.

%

\section{Worst-Case Behavior of the Classical PSG}
As a warm-up, we demonstrate in this section that, even on very simple problem instances, the classical PSG can exhibit worst-case behavior by terminating after a single iteration with arbitrarily high probability, due to the lack of subgradient information at the projected point. We state the following theorem.
\vspace{-3mm}
\begin{theorem}[One-step failure of the classical PSG in convex setting]\label{thm2.1}
	In the worst case, when using the classical step sizes \eqref{Nesterov} and \eqref{size3} for convex objectives, PSG with the initial point drawn at random from $\text{int}(\mathcal{X})$ fails after a single iteration with probability arbitrarily close to 1.
\end{theorem}
\begin{proof}
	The construction of worst case in convex setting is based on Example \ref{e1}. We demonstrate that, as $k_2/k_1 \to \infty ~\text{or}~ 0$, the classical projected subgradient method with step sizes \eqref{Nesterov} and \eqref{size3} terminates after a single iteration with probability arbitrarily close to 1.
	
	Here, without loss of generality, we restrict attention to the case $k_2 > k_1$ and define $\rho := k_2 / k_1$. Since the optimal solution of the problem in Example \ref{e1} is $(0,0)$, it follows from the definition of $R$ that, for the step sizes \eqref{Nesterov} and \eqref{size3}, $R$ is chosen as the length of the major axis of the ellipse, namely $R = \sqrt{r/k_1}$. Moreover, in Example \ref{e1}, the objective function admits no subgradient at any point on the boundary $k_1 x_1^2 + k_2 x_2^2 = r$. Therefore, it suffices to show that, for sufficiently large $\rho$, the one-step iterate of the PSG with the classical step sizes \eqref{Nesterov} and \eqref{size3} lies outside the feasible set with arbitrary high probability.
	
	This is equivalent to showing that, for $x_1 \in \text{int}(\mathcal{X}) := \{ x : k_1 x_1^2 + k_2 x_2^2 < r \}$ and any $g_1 \in \partial f(x_1)$, the point $x_1 - \alpha_1 g_1$ does not belong to $\operatorname{int}(X)$. Noticing that the first step size of the classical rules \eqref{Nesterov} and \eqref{size3} coincides, namely, $\alpha_1 = R / \|g_1\|$, and
	\begin{equation*}
		g_1 =
		\begin{pmatrix}
			\frac{k_1 x_1(1)}{\sqrt{r - k_1 x_1(1)^{2} - k_2 x_1(2)^{2}}} \\
			\frac{k_2 x_1(2)}{\sqrt{r - k_1 x_1(1)^{2} - k_2 x_1(2)^{2}}}
		\end{pmatrix},
	\end{equation*}
	we have 
	\begin{eqnarray}
		x_1 - \alpha_1 g_1 &=&
		\begin{pmatrix}
			x_1(1) - R\frac{k_1 x_1(1)}{\sqrt{k_1^{2} x_1(1)^{2} + k_2^{2} x_1(2)^{2}}} \\
			x_1(2) - R\frac{k_2 x_1(2)}{\sqrt{k_1^{2} x_1(1)^{2} + k_2^{2} x_1(2)^{2}}}
		\end{pmatrix}\nonumber\\
		&=&
		\begin{pmatrix}
			x_1(1) - \sqrt{\frac{r}{k1}}\frac{x_1(1)}{\sqrt{x_1(1)^{2} + \rho^{2} x_1(2)^{2}}} \\
			x_1(2) - \sqrt{\frac{r}{k1}}\frac{\rho x_1(2)}{\sqrt{x_1(1)^{2} + \rho^{2} x_1(2)^{2}}}
		\end{pmatrix}. \nonumber~~~~({\rm by ~the~ definition~ of~  R~ and ~ \rho})
	\end{eqnarray}
	Hence, it suffices to show that
	\begin{equation*}
		k_1 \left(x_1(1) - \frac{\sqrt{\frac{r}{k1}}x_1(1)}{\sqrt{x_1(1)^{2} + \rho^{2} x_1(2)^{2}}}\right)^2 + \rho k_1 \left(x_1(2) - \frac{\sqrt{\frac{r}{k1}}\rho x_1(2)}{\sqrt{x_1(1)^{2} + \rho^{2} x_1(2)^{2}}}\right)^2 \geq r.
	\end{equation*}
	Since 
	\begin{eqnarray}
		&&k_1 \left(x_1(1) - \frac{\sqrt{\frac{r}{k1}}x_1(1)}{\sqrt{x_1(1)^{2} + \rho^{2} x_1(2)^{2}}}\right)^2 + \rho k_1 \left(x_1(2) - \frac{\sqrt{\frac{r}{k1}}\rho x_1(2)}{\sqrt{x_1(1)^{2} + \rho^{2} x_1(2)^{2}}}\right)^2 \nonumber\\
		&=& k_1 \left(x_1(1)^2 + \frac{r}{k1}\frac{x_1(1)^2}{x_1(1)^{2} + \rho^{2} x_1(2)^{2}} - 2\sqrt{\frac{r}{k1}}\frac{x_1(1)^2}{\sqrt{x_1(1)^{2} + \rho^{2} x_1(2)^{2}}}\right) + \nonumber\\
		&& \rho k_1 \left(x_1(2)^2 + \frac{r}{k1}\frac{\rho^2x_1(2)^2}{x_1(1)^{2} + \rho^{2} x_1(2)^{2}} - 2\sqrt{\frac{r}{k1}}\frac{\rho x_1(2)^2}{\sqrt{x_1(1)^{2} + \rho^{2} x_1(2)^{2}}}\right) \nonumber\\
		&=& k_1 x_1(1)^2 + \rho k_1 x_1(2)^2 + \frac{rx_1(1)^{2} + r \rho^3x_1(2)^{2}}{x_1(1)^{2} + \rho^{2} x_1(2)^{2}} - 2\sqrt{rk_1\left(x_1(1)^{2} + \rho^{2} x_1(2)^{2}\right)}, \nonumber
	\end{eqnarray}
	it suffices to show that, when $\rho$ is sufficiently large, for any $x_1 \in \text{int}(\mathcal{X}) := \{ x : k_1 x(1)^2 + k_2 x(2)^2 < r \}$, the probability that
	\begin{equation}\label{eq3}
		k_1 x_1(1)^2 + \rho k_1 x_1(2)^2 + \dfrac{r x_1(1)^2 + r \rho^3 x_1(2)^2}{x_1(1)^2 + \rho^2 x_1(2)^2} - 2 \sqrt{ r k_1 \left( x_1(1)^2 + \rho^2 x_1(2)^2 \right) } \ge r
	\end{equation}
	is arbitrarily close to 1.
	
	For any  $x_1 \in \text{int}(\mathcal{X}) := \{ x : k_1 x(1)^2 + k_2 x(2)^2 < r \}$ and $x_1 \neq 0$, we parameterize points in the elliptical region as follows:
	\begin{equation*}
		\begin{cases}
			x_1(1) = \sqrt{\frac{c}{k_1}} \cos \theta, \\
			x_1(2) = \sqrt{\frac{c}{\rho k_1}} \sin \theta,
		\end{cases}
	\end{equation*}
	where $c \in (0, ~r)$ and $\theta \in [0, ~2\pi)$. Substituting the parameterization into the left hand side of \eqref{eq3}, we obtain
	\begin{eqnarray}
		&&k_1 x_1(1)^2 + \rho k_1 x_1(2)^2 + \dfrac{r x_1(1)^2 + r \rho^3 x_1(2)^2}{x_1(1)^2 + \rho^2 x_1(2)^2} - 2 \sqrt{ r k_1 \left( x_1(1)^2 + \rho^2 x_1(2)^2 \right) } \nonumber\\
		&=& c + \frac{r\cos^2(\theta) + r\rho^2\sin^2(\theta)}{\cos^2(\theta)+\rho\sin^2(\theta)} - 2\sqrt{rc\left(\cos^2(\theta)+\rho\sin^2(\theta)\right)} \nonumber\\
		&=& c + \frac{r + r(\rho^2-1)\sin^2(\theta)}{1+(\rho-1)\sin^2(\theta)} - 2\sqrt{rc\left(1+(\rho-1)\sin^2(\theta)\right)} \nonumber\\
		&=& c + r + \frac{r\rho(\rho-1)\sin^2(\theta)}{1+(\rho-1)\sin^2(\theta)} - 2\sqrt{rc\left(1+(\rho-1)\sin^2(\theta)\right)}. \nonumber
	\end{eqnarray}
	Hence, \eqref{eq3} is equivalent to
	\begin{equation}\label{eq4}
		c + \frac{r\rho(\rho-1)\sin^2(\theta)}{1+(\rho-1)\sin^2(\theta)} - 2\sqrt{rc\left(1+(\rho-1)\sin^2(\theta)\right)} \ge 0.
	\end{equation}
	For any arbitrarily small $\epsilon > 0$, as long as $\theta \in [\epsilon, \pi-\epsilon] \cup [\pi + \epsilon, 2\pi-\epsilon]$, there exists a positive constant $M$ such that \eqref{eq4} holds for all $\rho > M$. This is because
	\begin{eqnarray}
		&&c + \frac{r\rho(\rho-1)\sin^2(\theta)}{1+(\rho-1)\sin^2(\theta)} - 2\sqrt{rc\left(1+(\rho-1)\sin^2(\theta)\right)} \nonumber\\
		&\geq& c + \frac{r\rho(\rho-1)\sin^2(\epsilon)}{1+(\rho-1)} - 2\sqrt{rc\left(1+(\rho-1)\right)} \nonumber\\
		&=& c + r(\rho-1)\sin^2(\epsilon) - 2\sqrt{rc\rho}. \label{eq5}
	\end{eqnarray}
	Since the second term in \eqref{eq5} is of linear order in $\rho$ while the third term is of order $1/2$ in $\rho$, \eqref{eq5} can be made arbitrarily large when $\rho$ is sufficiently large. Therefore, there must exist some positive number $M$ such that for all $\rho > M$, \eqref{eq5} is greater than zero, and consequently \eqref{eq4} holds. Figure \ref{fig7} illustrates this.
	\begin{figure}[h]
		\centering
		\includegraphics[width=\linewidth]{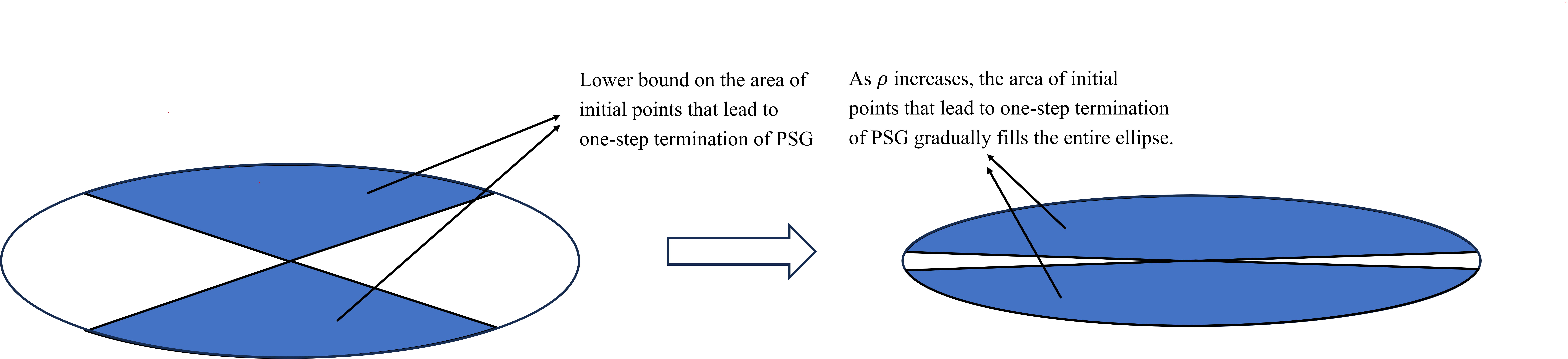}
		\vspace{2mm}
		\caption{Lower bound on the area of intial points that lead to one-step termination of PSG.}
		\label{fig7}
	\end{figure}

	This implies that when $\rho > M$, if the initial point $x_1$ is chosen from the region $\text{int}(\mathcal{X}) := \{ x : k_1 x(1)^2 + k_2 x(2)^2 < r \}$, and we denote by $\mathcal{X}_1$ the set of all $x_1$ that cause the PSG to terminate in one step, then
	\begin{equation*}
		\mathbb{P}\left(x_1 \in \mathcal{X}_1\right) = \dfrac{\text{area}(\mathcal{X}_1)}{\text{area}(\mathcal{X})} \geq \dfrac{\sqrt{\frac{r}{k_1}}\sqrt{\frac{r}{\rho k_1}}(\pi-2\epsilon)}{\sqrt{\frac{r}{k_1}}\sqrt{\frac{r}{\rho k_1}}\pi} = 1 - \frac{2\epsilon}{\pi}.
	\end{equation*}
	The theorem is thus proved by the arbitrariness of $\epsilon$.
\end{proof}

Theorem \ref{thm2.1} shows that, under the classical step sizes \eqref{Nesterov} and \eqref{size3} for convex objectives, PSG may terminate after a single iteration with arbitrarily high probability, even on simple problem instances. The following theorem establishes that the same conclusion holds for PSG in the strongly convex setting.
 
\begin{theorem}[One-step failure of the classical PSG in strongly convex setting]\label{thm2.2}
	In the worst case, when using the classical step size \eqref{size4} for strongly convex objectives, PSG with the initial point drawn at random from $\text{int}(\mathcal{X})$ fails after a single iteration with probability arbitrarily close to 1.
\end{theorem}
\begin{proof}
	The construction of worst case in strongly convex setting is based on Example \ref{e3}. We show that, when the problem dimension $n$ is sufficiently large, the classical projected subgradient method with step sizes \eqref{size4} terminates after a single iteration with probability arbitrarily close to 1.
	
	Note that the negative entropy function $f(x) = \sum_{i=1}^n x(i)\log x(i)$ is $\mu$-strongly convex on the hypercube $\mathcal{X} = \{ x \in \mathbb{R}_+^n : 0 \le x(i) \le B, B > 0 \}$ for any $0 < \mu \le 1/B$. Here we consider the case $B > 1$, in which the strong convexity parameter satisfies $0 < \mu \le 1/B < 1$. This strong convexity property also holds for $f(x)$ with respect to each coordinate $x(i)$. Since $f(x)$ admits no subgradient at any boundary point where $x(i)=0$, it suffices to show that, when $n$ is sufficiently large, a single iteration of PSG with the classical step size \eqref{size4} produces, with high probability, a point for which at least one coordinate satisfies $x(i)=0$.
	
	We assume that the initial point $x_1$ is generated randomly from $\text{int}(\mathcal{X})$. Since $f(x)$ is also $\mu$-strongly convex with respect to each coordinate $x(i)$ for $0 < \mu \le 1/B$, and is separable across coordinates, PSG with step size \eqref{size4} applied to the minimization of $f(x)$ is essentially equivalent to running PSG with step size \eqref{size4} on $n$ independent problems, each initialized at a different random starting point $x_1(i)$, namely,
	\begin{equation}
		\begin{alignedat}{2}
			&\min_{x(i)} && ~ x(i) \log x(i)\\
			&\ST&&0 \le x(i) \le B.
		\end{alignedat}
		\label{model}
	\end{equation}
	Therefore, we first analyze the one-dimensional case, considering PSG with step size \eqref{size4} applied to problem \eqref{model}, and study the event that a single iteration leads to $x_2(i)=0$. For the step size \eqref{size4}, we have $\alpha_1 = 1/\mu$, and the gradient of problem \eqref{model} is $1 + \log x(i)$. Therefore, this is equivalent to considering when
	\begin{equation}\label{equ15}
		x_1(i) - \frac{1}{\mu}\left(1 + \log x_1(i)\right) \le 0.
	\end{equation}
	It suffices to introduce the one-dimensional auxiliary function $q(x) = x - \frac{1}{\mu}(1+\log x)$ and analyze the condition under which $q(x) \le 0$ on the interval $[0,B]$. Note that $q'(x) = \frac{\mu x - 1}{\mu x}$. For $x \in [0,B]$, we have $q'(x) \le 0$ since $0 < \mu \le 1/B$. Hence, $q(x)$ is monotonically decreasing on $[0,B]$. Observe that $q(1/e) = 1/e > 0$ and $q(1) = 1 - 1/\mu < 0$ because $0 < \mu \le 1/B < 1$. Therefore, there exists a unique $p \in (0,B)$ such that $q(p)=0$. By the monotonicity of $q(x)$, it follows that $q(x) \le 0$ for all $x \in [p,B]$. This implies that, for $x_1(i)$ generated randomly in $(0,B)$, \eqref{equ15} holds with probability $(B-p)/B = 1 - p/B >0$.
	
	Consequently, for an initial point $x_1$ generated randomly in $\text{int}(\mathcal{X})$, since the coordinates $x_1(i)$ are independent, the probability that a single iteration of PSG with the classical step size \eqref{size4} yields at least one coordinate satisfying $x_2(i)=0$ is
	\begin{equation*}
		\mathbb{P}(\text{PSG with step size \eqref{size4} terminate after 1 iteration}) = 1-\left(\frac{p}{B}\right)^n.
	\end{equation*}
	As $n \to \infty$, the probability that PSG terminates after a single iteration converges to one. This completes the proof.
\end{proof}

Theorems \ref{thm2.1} and \ref{thm2.2} jointly indicate that, whether in the convex or strongly convex setting, the classical PSG can exhibit worst-case behavior even on simple examples (Examples \ref{e1} and \ref{e3}), namely terminating in a single step with high probability. Example \ref{e2} provides another interesting scenario where PSG fails. In Example \ref{e2}, the optimal point is located at $(0,0)$. However, on the boundary $x(1) = 0$ near the optimal point (excluding the optimal point itself), the objective function has no subgradient. This implies that when an iterate gets sufficiently close to the optimum, it is highly likely to be projected onto the boundary near the optimal point, causing previous progress to be largely undone. In Section 6, we will further illustrate these PSG failure scenarios through numerical experiments.

\section{Subgradient Gliding Method for Convex Function}

In this section, we formally present the subgradient gliding method. The subgradient gliding method iterates according to the following scheme:
\begin{equation*}
	\left\{	\begin{aligned}
		y_{s+1}&=x_s-\alpha_s g_s, ~\text{where } g_s\in\partial f(x_s),\\
		z_{s+1}&=  \argmin_{x\in\mathcal{X}}\Vert x-y_{s+1} \Vert,\\
		x_{s+1}&= (1-\beta_s)x_s +\beta_s z_{s+1}.	\end{aligned} \right.
\end{equation*}
The procedure is illustrated in Figure~\ref{fig1}. Compared with the classical PSG, the proposed method introduces an additional gliding step after projection. Instead of directly updating the iterate to the projection point $z_{s+1}$, the method moves along the projected displacement direction $z_{s+1}-x_s$ with a carefully chosen step size. This controlled update prevents the iterate from landing on boundary points where subgradients may fail to exist, while still preserving the geometric information encoded by the projection. Motivated by this interior-preserving evolution where the iterates advance by smoothly following the projected displacement without directly hitting the boundary, we term the proposed algorithm the \textit{subgradient gliding method}. Throughout the paper, $\alpha_s$ and $\beta_s$ denote the subgradient step size and the gliding step size, respectively.
\begin{figure}[h]
	\centering
	\includegraphics[width=0.6\linewidth]{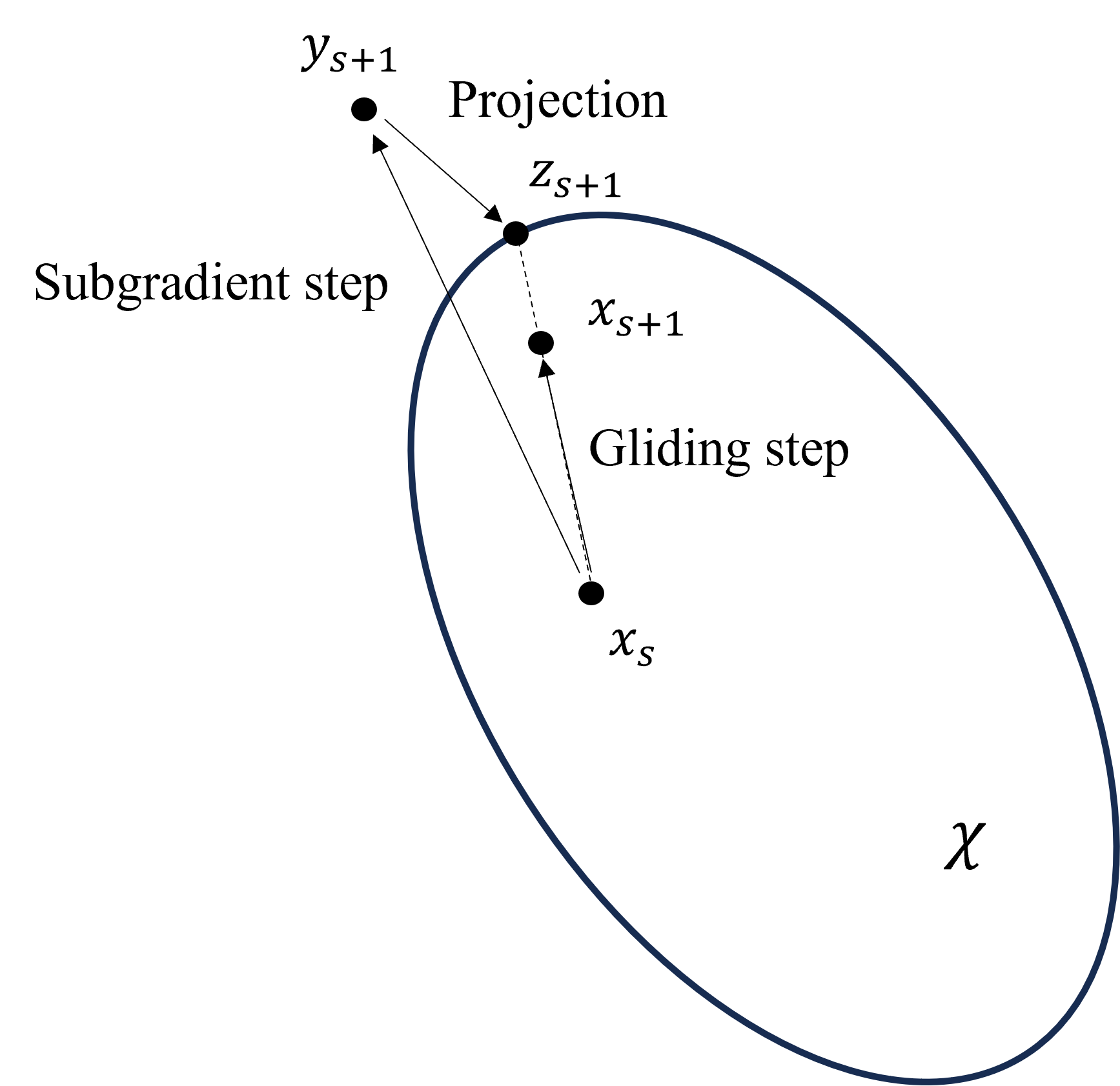}
	\vspace{2mm}
	\caption{Iteration scheme of subgradient gliding method.}
	\label{fig1}
\end{figure}

We now proceed to prove that for general convex functions, the joint design of subgradient and gliding step sizes enables the achievement of the optimal ergodic convergence rate, while simultaneously ensuring that all iterates remain strictly within the interior of the convex set.


\begin{theorem}\label{thm2}
	Suppose that the following conditions hold:
	\begin{description}
		\item[(C1)] The sequence $\left\{ \frac{w_s}{\alpha_s \beta_s} \right\}_{s \ge 1}$ is monotonically nondecreasing;
		\item[(C2)] The gliding step size satisfies $0 < \beta_s < 1$ for all $s$.
	\end{description}
	If the initial point $x_1 \in \operatorname{int}(\mathcal{X})$, then all iterates $x_s \in \operatorname{int}(\mathcal{X})$, and the following ergodic convergence bound holds:
	\begin{equation}\label{Convergence}
		f\left( \frac{\sum_{s=1}^t w_s x_s}{\sum_{s=1}^t w_s} \right) - f(x^*)
		\leq \frac{1}{\sum_{s=1}^{t}w_s}
		\left(
		\frac{R^2 w_t}{2 \alpha_t \beta_t}
		+ \sum_{s=1}^{t} \frac{w_s \alpha_s}{2} \Vert g_s \Vert^2
		\right),
	\end{equation}
	where $w_s \ge 0$ denotes the ergodic weights.
\end{theorem}

\begin{proof}
	First, we prove that \eqref{Convergence} holds.
	\begin{eqnarray}
		f(x_s)-f(x^*)&\leq& g_s^T(x_s-x^*) \nonumber~~~~({\rm by ~the~ definition~ of~  subgradient}) \\
		&=& \frac{1}{\alpha_s}(x_s-y_{s+1})^T(x_s-x^*) \nonumber\\
		&=& \frac{1}{2\alpha_s}(\Vert x_s-y_{s+1} \Vert^2+\Vert x_s-x^* \Vert^2-\Vert y_{s+1}-x^* \Vert^2)\label{eq1}\\
		&=& \frac{1}{2\alpha_s}(\Vert x_s-x^* \Vert^2-\Vert y_{s+1}-x^* \Vert^2)+\frac{\alpha_s}{2}\Vert g_s \Vert^2 \nonumber\\
		&\leq& \frac{1}{2\alpha_s}(\Vert x_s-x^* \Vert^2-\Vert z_{s+1}-x^* \Vert^2)+\frac{\alpha_s}{2}\Vert g_s \Vert^2\label{ineq1}\\
		&=& \frac{1}{2\alpha_s}(\Vert x_s-x^* \Vert^2-\Vert \frac{1}{\beta_s}x_{s+1}-\frac{1-\beta_s}{\beta_s}x_s-x^* \Vert^2)+\frac{\alpha_s}{2}\Vert g_s \Vert^2\nonumber\\
		&=& \frac{1}{2\alpha_s}(\Vert x_s-x^* \Vert^2-\Vert \frac{1}{\beta_s}(x_{s+1}-x^*)-\frac{1-\beta_s}{\beta_s}(x_s-x^*) \Vert^2)+\frac{\alpha_s}{2}\Vert g_s \Vert^2\nonumber\\
		&=& \frac{1}{2\alpha_s}(\frac{2\beta_s-1}{\beta_s^2}\Vert x_s-x^* \Vert^2-\frac{1}{\beta_s^2}\Vert x_{s+1}-x^* \Vert^2\nonumber\\
		&&+\frac{2(1-\beta_s)}{\beta_s^2}(x_{s+1}-x^*)^T(x_s-x^*))+\frac{\alpha_s}{2}\Vert g_s \Vert^2\nonumber\\
		&=& \frac{1}{2\alpha_s}\cdot\frac{1}{\beta_s^2}(\Vert x_s-x^* \Vert^2-\Vert x_{s+1}-x^* \Vert^2)\nonumber\\
		&&+\frac{1}{2\alpha_s}\cdot\frac{2(1-\beta_s)}{\beta_s^2}\left((x_{s+1}-x^*)^T(x_s-x^*)-\Vert x_s-x^* \Vert^2\right)+\frac{\alpha_s}{2}\Vert g_s \Vert^2\nonumber\\
		&=& \frac{1}{2\alpha_s}\cdot\frac{1}{\beta_s^2}(\Vert x_s-x^* \Vert^2-\Vert x_{s+1}-x^* \Vert^2)\nonumber\\
		&&+\frac{1}{2\alpha_s}\cdot\frac{2(1-\beta_s)}{\beta_s^2}(x_{s+1}-x_s)^T(x_s-x^*)+\frac{\alpha_s}{2}\Vert g_s \Vert^2\nonumber\\
		&=& \frac{1}{2\alpha_s}\cdot\frac{1}{\beta_s^2}(\Vert x_s-x^* \Vert^2-\Vert x_{s+1}-x^* \Vert^2)+\frac{1}{2\alpha_s}\cdot\frac{1-\beta_s}{\beta_s^2}(\Vert x_{s+1}-x^* \Vert^2\nonumber\\
		&&-\Vert x_s-x^* \Vert^2-\Vert x_{s+1}-x_s \Vert^2)+\frac{\alpha_s}{2}\Vert g_s \Vert^2\label{eq2}\\
		&=& \frac{1}{2\alpha_s}\cdot\frac{1}{\beta_s}(\Vert x_s-x^* \Vert^2-\Vert x_{s+1}-x^* \Vert^2)\nonumber\\
		&&-\frac{1}{2\alpha_s}\cdot\frac{1-\beta_s}{\beta_s^2}\Vert x_{s+1}-x_s \Vert^2+\frac{\alpha_s}{2}\Vert g_s \Vert^2\nonumber\\
		&\leq& \frac{1}{2\alpha_s}\cdot\frac{1}{\beta_s}(\Vert x_s-x^* \Vert^2-\Vert x_{s+1}-x^* \Vert^2)+\frac{\alpha_s}{2}\Vert g_s \Vert^2 ,\label{ineq2}
	\end{eqnarray}
	where \eqref{eq1} and \eqref{eq2} is derived from the identity $2a^Tb=\Vert a\Vert^2+\Vert b\Vert^2-\Vert a-b\Vert^2=\Vert a+b\Vert^2-\Vert a\Vert^2-\Vert b\Vert^2$, and \eqref{ineq1} holds due to the fact that
	\begin{equation*}
		\Vert y_{s+1}-x^* \Vert^2\geqslant \Vert z_{s+1}-x^* \Vert^2,
	\end{equation*}
	which is a direct consequence of the projection theorem \cite[Lemma 3.1]{bubeck2015convex}. 
	
	Hence,
	\begin{eqnarray}
		&& \left(\sum_{s=1}^{t}w_s\right)\left(f\left( \frac{\sum_{s=1}^t w_s x_s}{\sum_{s=1}^t w_s} \right) - f(x^*)\right)\nonumber\\
		&\leq&\sum_{s=1}^t w_s(f(x_s)-f(x^*))  ~~~~({\rm Jensen's~ inequality})\nonumber\\
		&\leq&\sum_{s=1}^t \frac{w_s}{2\alpha_s }\cdot\frac{1}{\beta_s}(\Vert x_s-x^* \Vert^2-\Vert x_{s+1}-x^* \Vert^2)+\sum_{s=1}^t\frac{w_s\alpha_s}{2}\Vert g_s \Vert^2 ~~~~({\rm by~ \eqref{ineq2}}) \nonumber\\
		&=& \frac{w_1}{2\alpha_1\beta_1}\Vert x_1-x^* \Vert^2+\frac{1}{2}\sum_{s=2}^t\left(\frac{w_s}{ \alpha_s\beta_s }-\frac{w_{s-1}}{ \alpha_{s-1}\beta_{s-1}}\right)\Vert x_s-x^* \Vert^2\nonumber\\
		&& -\frac{w_t^{(k)}}{ 2\alpha_{t}\beta_t }\Vert x_{t+1}-x^* \Vert^2+\sum_{s=1}^t\frac{w_s\alpha_s}{2}\Vert g_s \Vert^2\nonumber\\
		&\leq& R^2 \left( \frac{w_1}{2\alpha_1\beta_1}+\frac{1}{2}\sum_{s=2}^t\left(\frac{w_s}{ \alpha_s\beta_s }-\frac{w_{s-1}}{ \alpha_{s-1}\beta_{s-1}}\right) \right)+\sum_{s=1}^t\frac{w_s\alpha_s}{2}\Vert g_s \Vert^2 \label{ineq3}\\
		&=& \frac{R^2 w_t}{2\alpha_t\beta_t}+\sum_{s=1}^t\frac{w_s\alpha_s}{2}\Vert g_s \Vert^2, \nonumber
	\end{eqnarray}
	where \eqref{ineq3} holds since $\frac{w_s}{\alpha_s\beta_s}$ is monotonically nondecreasing. We next show that if $0 < \beta_s <1$ and the initial point $x_1 \in \text{int}(\mathcal{X})$, all iterates $x_s \in \text{int}(\mathcal{X})$.
	
	Suppose that $x_s \in \text{int}(\mathcal{X})$. Since the projection point $z_{s+1}\in\mathcal{X}$ and $0 < \beta_s <1$, then by \cite[Theorem 6.1]{Rockafellar}, we have 
	\begin{equation*}
		x_{s+1} = (1-\beta_s)x_s + \beta_s z_{s+1} \in \text{int}(\mathcal{X}).
	\end{equation*}
	Since $x_1 \in \text{int}(\mathcal{X})$, it follows by induction that $x_s \in \text{int}(\mathcal{X})$ for all $s$. The proof is complete.
\end{proof}

Theorem \ref{thm2} provides fundamental guidelines for the joint design of the subgradient step size $\alpha_s$, gliding step size $\beta_s$, and ergodic weights $w_s$. The following theorem shows that, by adopting the classical subgradient step sizes $\alpha_s$ and ergodic weights $w_s$ in PSG \cite{bubeck2015convex,beck2017first,nesterov2018lectures,Zhu, Xia}, the subgradient gliding method can achieve the optimal ergodic convergence rate through a simple design of the gliding step size $\beta_s$.

\begin{theorem}[Gliding step size design under classical PSG settings]\label{thm3}
	Suppose that the following conditions on $\beta_s$ hold:
	\begin{description}
		\item[(C1)] The sequence $\left\{ \frac{w_s}{\alpha_s \beta_s} \right\}_{s \ge 1}$ is monotonically nondecreasing;
		\item[(C2$^+$)] $0 < c \leq \beta_s <1$ holds for all $s$ where c is an arbitrary positive constant.
	\end{description}
	For any fixed $k\geq -1$, define $w_s^{(k)}$ as  
	\begin{equation}\label{weight}
		w_s^{(k)} = 
		\begin{cases} 
			1/\alpha_s^k, & -1\leq k \leq 0, \\ 
			s^{\frac{k}{2}}, & ~~k > 0.
		\end{cases}
	\end{equation}
	If the initial point $x_1 \in \operatorname{int}(\mathcal{X})$, then all iterates $x_s \in \operatorname{int}(\mathcal{X})$, and
	
	(1) Subgradient gliding method with constant subgradient step-size \eqref{size1} satisfies (where ergodic weight $w_s$ is constant)   
	\begin{equation}\label{WeakConvergence}
		f\left(\frac{\sum_{s=1}^t x_s}{t}\right)-f(x^*)\leq \frac{RL}{c\sqrt{t}}.
	\end{equation}
	
	(2) Subgradient gliding method with time-varying subgradient step-size \eqref{size2} satisfies (where ergodic weight $w_s = w_s^{(k)}$ for any fixed $k \geq -1$)
	\begin{equation}\label{WeakConvergence1}
		f\left( \frac{\sum_{s=1}^t w_s^{(k)} x_s}{\sum_{s=1}^t w_s^{(k)}} \right) - f(x^*)
		\leq \frac{t^{\frac{k+1}{2}} +\sum_{s=1}^{t} s^{\frac{k-1}{2}}}{2c\sum_{s=1}^{t} s^{\frac{k}{2}}} RL.
	\end{equation}
	
	(3) Subgradient gliding method with subgradient-normalized time-varying subgradient step-size \eqref{Nesterov} satisfies (where ergodic weight $w_s = \alpha_s$)
	\begin{equation}\label{WeakConvergence2}
		f\left(\frac{\sum_{s=1}^t \alpha_s x_s}{\sum_{s=1}^t \alpha_s}\right)-f(x^*)\leq \frac{2RL+RL\log t}{4c(\sqrt{t+1}-1)}.
	\end{equation}
	
	(4) Subgradient gliding method with subgradient step-size family \eqref{size3} satisfies (where ergodic weight $w_s = w_s^{(k)}$ for any fixed $k \geq -1$)   
	\begin{equation}\label{WeakConvergence3}
		f\left( \frac{\sum_{s=1}^t w_s^{(k)} x_s}{\sum_{s=1}^t w_s^{(k)}} \right) - f(x^*)
		\leq \frac{t^{\frac{k+1}{2}} +\sum_{s=1}^{t} s^{\frac{k-1}{2}}}{2c\sum_{s=1}^{t} s^{\frac{k}{2}}} R\max\limits_{s=1,\dots,t}{\|g_s\|}.
	\end{equation}
\end{theorem}

\begin{proof}
	We first prove case (1), where $w_s$ is constant. By \eqref{Convergence},
	\begin{eqnarray}
		f\left(\frac{\sum_{s=1}^t x_s}{t}\right)-f(x^*) 
		&\leq& \frac{1}{t}\left(\frac{R^2}{2\alpha_t\beta_t}+\sum_{s=1}^{t}\frac{\alpha_s}{2}\Vert g_s \Vert^2\right)\nonumber\\
	    &\leq&\frac{1}{t}\left(\frac{RL\sqrt{t}}{2\beta_t}+\frac{RL\sqrt{t}}{2}\right)~~~~({\rm by~ \eqref{size1}})\nonumber\\
		&\leq& \frac{RL}{c\sqrt{t}}.~~~~({\rm by~ 0 < c \leq \beta_s <1})\nonumber
	\end{eqnarray}
	We now prove case (2). For the time-varying step size \eqref{size2}, using the weight \eqref{weight} is equivalent to setting $w_s = s^{k/2}$ for any fixed $k \geq -1$. Hence, from \eqref{Convergence},
	\begin{eqnarray}
		f\left( \frac{\sum_{s=1}^t w_s^{(k)} x_s}{\sum_{s=1}^t w_s^{(k)}} \right) - f(x^*)
		&\leq& \frac{1}{\sum_{s=1}^{t} s^{\frac{k}{2}}}\left(\frac{R^2 t^{\frac{k}{2}}}{2\alpha_t\beta_t}+\sum_{s=1}^{t}\frac{s^{\frac{k}{2}}\alpha_s}{2}\Vert g_s \Vert^2\right)\nonumber\\
		&\leq&\frac{1}{\sum_{s=1}^{t} s^{\frac{k}{2}}}\left(\frac{RL t^{\frac{k+1}{2}}}{2\beta_t}+\sum_{s=1}^{t}\frac{RLs^{\frac{k-1}{2}}}{2}\right)~~~~({\rm by~ \eqref{size2}})\nonumber\\
		&\leq& \frac{t^{\frac{k+1}{2}} +\sum_{s=1}^{t} s^{\frac{k-1}{2}}}{2c\sum_{s=1}^{t} s^{\frac{k}{2}}} RL.~~~~({\rm by~ 0 < c \leq \beta_s <1})\nonumber
	\end{eqnarray}
	We now turn to the proof of case (3). Substituting the time-varying step size \eqref{Nesterov} and the weight $w_s = \alpha_s$ into \eqref{Convergence} yields
	\begin{eqnarray}
		f\left(\frac{\sum_{s=1}^t \alpha_s x_s}{\sum_{s=1}^t \alpha_s}\right)-f(x^*)
		&\leq& \frac{1}{\sum_{s=1}^{t} \alpha_s}\left(\frac{R^2}{2\beta_t}+\sum_{s=1}^{t}\frac{\alpha_s^2}{2}\Vert g_s \Vert^2\right)\nonumber\\
		&=& \frac{1}{\sum_{s=1}^{t} \frac{R}{\|g_s\|\sqrt{s}}}\left(\frac{R^2}{2\beta_t}+\sum_{s=1}^{t}\frac{R^2}{2s}\right)~~~~({\rm by~ \eqref{Nesterov}})\nonumber\\
		&\leq& \frac{RL}{\sum_{s=1}^{t} \frac{1}{\sqrt{s}}}\left(\frac{1}{2\beta_t}+\sum_{s=1}^{t}\frac{1}{2s}\right)\nonumber\\
		&\leq& \frac{2RL+RL\log t}{4c(\sqrt{t+1}-1)}, \label{ineq5}
	\end{eqnarray}
	where \eqref{ineq5} holds since $\sum_{s=1}^{t}\frac{1}{s} \le 1 + \log t$, $\sum_{s=1}^{t} \frac{1}{\sqrt{s}} \geq 2(\sqrt{t+1}-1)$ and $0 < c \leq \beta_s <1$.
	
	We conclude with the proof of case (4). By definition of $G_s$ in \eqref{G_s}, we also have the equivalent form of $G_s$ as follows:
	\begin{equation}\label{G_s1}
		G_s = \max\limits_{j=1,\dots,s}{\{\|g_j\|j^{\frac{1-a}{2}}\}}.
	\end{equation}
	Considering the family of time-varying step sizes \eqref{size3} and the ergodic weights \eqref{weight} and substituting them into \eqref{Convergence}, for $-1 \le k \le 0$, we obtain
	\begin{eqnarray}
		&&f\left( \frac{\sum_{s=1}^t w_s^{(k)} x_s}{\sum_{s=1}^t w_s^{(k)}} \right) - f(x^*)\nonumber \\
		&\leq& \frac{\frac{R^2}{2\alpha_t^{k+1}\beta_t}+\sum_{s=1}^t\frac{\alpha_s^{1-k}}{2}\Vert g_s \Vert^2}{\sum_{s=1}^t \alpha_s^{-k}}\nonumber \\
		&\leq&
		\frac{R}{2c}\frac{\left({G_t t^{\frac{a}{2}}}\right)^{1+k}+\sum_{s=1}^t\Vert g_s \Vert^2 \left({G_s s^{\frac{a}{2}}}\right)^{k-1}}{\sum_{s=1}^t \left({G_s s^{\frac{a}{2}}}\right)^{k}} ~~~~({\rm by~ \eqref{size3}}~and ~0 < c \leq \beta_s <1)\nonumber \\
		&\leq&
		\frac{R}{2c}\frac{\left({\max\limits_{s=1,\cdots,t}{\|g_s\|}}\right)^{1+k}t^{\frac{k+1}{2}}+\sum_{s=1}^t\Vert g_s \Vert^{1+k}s^{\frac{k-1}{2}}}{\left({\max\limits_{s=1,\cdots,t}{\|g_s\|}}\right)^{k}\sum_{s=1}^t s^{\frac{k}{2}}} \label{s1}\\
		&\leq&
		\frac{t^{\frac{k+1}{2}} +\sum_{s=1}^{t} s^{\frac{k-1}{2}}}{2c\sum_{s=1}^{t} s^{\frac{k}{2}}} R\max\limits_{s=1,\dots,t}{\|g_s\|}, \nonumber
	\end{eqnarray}
	where \eqref{s1} holds by definition \eqref{G_s}, \eqref{G_s1} and $0 \leq a \leq 1$ when $-1 \leq k \leq 0$. And for $k > 0$, we have
	\begin{eqnarray}
		&&f\left( \frac{\sum_{s=1}^t w_s^{(k)} x_s}{\sum_{s=1}^t w_s^{(k)}} \right) - f(x^*)\nonumber \\
		&\leq& \frac{\frac{R^2 t^{\frac{k}{2}}}{2\alpha_t\beta_t}+\sum_{s=1}^t\frac{s^{\frac{k}{2}}\alpha_s}{2}\Vert g_s \Vert^2}{\sum_{s=1}^t s^{\frac{k}{2}}} \nonumber \\
		&\leq& \frac{R}{2c}\frac{t^{\frac{k}{2}}G_t t^{\frac{a}{2}}+\sum_{s=1}^t\frac{s^{\frac{k}{2}}\Vert g_s \Vert^2}{G_s s^{\frac{a}{2}}}}{\sum_{s=1}^t s^{\frac{k}{2}}} ~~~~({\rm by~ \eqref{size3}~ and ~0 < c \leq \beta_s <1})\nonumber \\
		&\leq& \frac{R}{2c}\frac{t^{\frac{k+1}{2}}\max\limits_{s=1,\cdots,t}{\|g_s\|}+\sum_{s=1}^t\Vert g_s \Vert s^{\frac{k-1}{2}}}{\sum_{s=1}^t s^{\frac{k}{2}}}\label{s2} \\
		&\leq&
		\frac{t^{\frac{k+1}{2}} +\sum_{s=1}^{t} s^{\frac{k-1}{2}}}{2c\sum_{s=1}^{t} s^{\frac{k}{2}}} R\max\limits_{s=1,\dots,t}{\|g_s\|}, \nonumber
	\end{eqnarray}
	where \eqref{s2} follows from definition \eqref{G_s}, \eqref{G_s1} and $0 \leq a \leq 1$ when $k > 0$. The proof is complete.
\end{proof}

\begin{remark}
	In cases (2) and (4), when $k = -1$, the subgradient gliding method achieves a convergence rate of $\mathcal{O}\left(\log t / \sqrt{t}\right)$. For $k > -1$, the method attains the optimal ergodic convergence rate of $\mathcal{O}\left(1 / \sqrt{t}\right)$.
\end{remark}

\begin{remark}
	The conditions of gliding step size $\beta_s$ in (C1) and (C2$^+$) of Theorem \ref{thm3} are mild and straightforward to satisfy. A simplest constructive choice is $\beta_s\equiv \beta, ~s=1,\cdots,t$ for any fixed constant $\beta \in (0,1)$.
\end{remark}

\begin{corollary}\label{PSG_special1}
	Setting $\beta = 1$ recovers the classical PSG as a special case of the subgradient gliding framework, in which case the interior-point property of the iterates is no longer guaranteed.
\end{corollary}

Corollary \ref{PSG_special1} shows that the classical PSG could be view as a special case of the subgradient gliding method. More importantly, the admissible step size rules for the subgradient gliding method substantially enlarges the step size design space while maintaining theoretical convergence guarantees.

\begin{remark}
	Provided that the gliding step size $\beta_s$ satisfies the conditions in (C1) and (C2$^+$), the convergence guarantees stated above for the different choices of subgradient step size $\alpha_s$, namely cases (1), (2), (3), and (4), all remain valid. This implies that even when the subgradient step size $\alpha_s$ and ergodic weight $w_s$ have been fixed, the gliding step size $\beta_s$ can be flexibly designed (e.g., via line search) as long as (C1) and (C2$^+$) are met, thereby preserving the theoretical guarantees while potentially yielding more pronounced practical gains.
\end{remark}


Theorem \ref{thm3} only establishes design guidelines for the gliding step size under several common choices of time-varying step sizes $\alpha_s$ and ergodic weights $w_s$ in PSG. In fact, guided by (C1) and (C2$^+$), further joint designs of the subgradient step size $\alpha_s$, gliding step size $\beta_s$, and ergodic weight $w_s$ are possible. We provide a practical joint design that automatically satisfies (C1) (thereby condition (C1) is even no longer needed), achieves the optimal ergodic convergence rate, and significantly relaxes the admissible choices of the gliding step size $\beta_s$. This leads to the following Theorem.

\begin{theorem}[A practical joint design that eliminates (C1)]\label{cor:joint_design}
	Suppose that the gliding step size $\beta_s$ satisfies condition (C2$^+$).
	If the initial point $x_1 \in \operatorname{int}(\mathcal{X})$, then all iterates $x_s \in \operatorname{int}(\mathcal{X})$, and \footnote{Same convergence result for $\min_{s=1,\ldots,t} f(x_s)$ in case (1-4) can be obtained by similar analysis.}
	
	(1) For the subgradient gliding method with the constant subgradient step size \eqref{size1} and  ergodic weights $w_s = \beta_s$, it holds that
	\begin{equation}\label{WeakConvergence1pd}
		f\left( \frac{\sum_{s=1}^t \beta_s x_s}{\sum_{s=1}^t \beta_s} \right) - f(x^*)
		\leq \frac{RL}{c\sqrt{t}}.
	\end{equation}
	
	(2) For the subgradient gliding method with the time-varying subgradient step size \eqref{size2} and ergodic weights $w_s = \beta_s$, it holds that
	\begin{equation}\label{WeakConvergence2pd}
		f\left( \frac{\sum_{s=1}^t \beta_s x_s}{\sum_{s=1}^t \beta_s} \right) - f(x^*)
		\leq \frac{3RL}{2c\sqrt{t}}.
	\end{equation}
	
	(3) For the subgradient gliding method with the subgradient-normalized time-varying subgradient step size \eqref{Nesterov} and ergodic weights $w_s = \alpha_s \beta_s$, it holds that
	\begin{equation}\label{WeakConvergence3pd}
		f\left(\frac{\sum_{s=1}^t \alpha_s\beta_s x_s}{\sum_{s=1}^t \alpha_s\beta_s}\right)-f(x^*)\leq \frac{2RL+RL\log t}{4c(\sqrt{t+1}-1)}.
	\end{equation}
	
	(4) For the subgradient gliding method with the subgradient step size family \eqref{size3} and ergodic weights $w_s = \beta_s$, it holds that
	\begin{equation}\label{WeakConvergence4pd}
		f\left( \frac{\sum_{s=1}^t \beta_s x_s}{\sum_{s=1}^t \beta_s} \right) - f(x^*)
		\leq \frac{3R}{2c\sqrt{t}}\cdot \max_{s=1,\cdots,t}\|g_s\|.
	\end{equation}
\end{theorem}

\begin{proof}
	We first prove that these joint designs satisfy (C1) in Theorem \ref{thm3}.
	\begin{enumerate}
		\item When the subgradient step sizes \eqref{size1}, \eqref{size2} and \eqref{size3} are employed, the weak ergodic weights $w_s = \beta_s$ automatically satisfies (C1) in Theorem \ref{thm3}. In this scenerio, $\frac{w_s}{\alpha_s \beta_s} = \frac{1}{\alpha_s}$. Since the step sizes \eqref{size1}, \eqref{size2} and \eqref{size3} are nonincreasing, $\frac{1}{\alpha_s}$ is nondecreasing, which naturally satisfies (C1).
		\item When the subgradient step size \eqref{Nesterov} is employed, the weak ergodic weights $w_s = \alpha_s\beta_s$ also satisfies (C1) in Theorem \ref{thm3}. In this scenerio, $\frac{w_s}{\alpha_s \beta_s} = 1$, which is also nondecreasing. 
	\end{enumerate}
	Hence, as long as the gliding step size $\beta_s$ satisfies (C2$^+$), Theorem \ref{thm2} holds for case (1), (2), (3) and (4) of Theorem \ref{cor:joint_design}.
	
	We start by proving case (1) of Theorem \ref{cor:joint_design}. By \eqref{Convergence}, we have
	\begin{eqnarray}
		f\left( \frac{\sum_{s=1}^t \beta_s x_s}{\sum_{s=1}^t \beta_s} \right) - f(x^*) 
		&\leq& \frac{1}{\sum_{s=1}^{t}\beta_s}
		\left(
		\frac{R^2}{2 \alpha_t}
		+ \sum_{s=1}^{t} \frac{\alpha_s \beta_s}{2} \Vert g_s \Vert^2
		\right)\nonumber\\
		&\leq&\frac{1}{ct}\left(\frac{RL\sqrt{t}}{2}+\sum_{s=1}^{t} \frac{RL}{2\sqrt{t}}\right)~~~~({\rm by~ (C2^+)})\nonumber\\
		&\leq& \frac{RL}{c\sqrt{t}}.\nonumber
	\end{eqnarray}
	
	We proceed to prove case (2) of Theorem \ref{cor:joint_design}. By \eqref{Convergence}, we have
	\begin{eqnarray}
		f\left( \frac{\sum_{s=1}^t \beta_s x_s}{\sum_{s=1}^t \beta_s} \right) - f(x^*) 
		&\leq& \frac{1}{\sum_{s=1}^{t}\beta_s}
		\left(
		\frac{R^2}{2 \alpha_t}
		+ \sum_{s=1}^{t} \frac{\alpha_s \beta_s}{2} \Vert g_s \Vert^2
		\right)\nonumber\\
		&\leq&\frac{1}{ct}\left(\frac{RL\sqrt{t}}{2}+\sum_{s=1}^{t} \frac{RL}{2\sqrt{s}}\right)~~~~({\rm by~ (C2^+)})\nonumber\\
		&\leq& \frac{3RL}{2c\sqrt{t}},\label{eq9}
	\end{eqnarray}
	where \eqref{eq9} holds since $\sum_{s=1}^{t} \frac{1}{\sqrt{s}} \le 2\sqrt{t}$.
	
	Now we prove case (3) of Theorem \ref{cor:joint_design}. By \eqref{Convergence}, we have
	\begin{eqnarray}
		f\left( \frac{\sum_{s=1}^t \alpha_s\beta_s x_s}{\sum_{s=1}^t \alpha_s\beta_s} \right) - f(x^*) 
		&\leq& \frac{1}{\sum_{s=1}^{t}\alpha_s\beta_s}
		\left(
		\frac{R^2}{2}
		+ \sum_{s=1}^{t} \frac{\beta_s \alpha_s^2}{2} \Vert g_s \Vert^2
		\right)\nonumber\\
		&\leq&\frac{1}{c\sum_{s=1}^{t}\frac{R}{\|g_s\|\sqrt{s}}}\left(\frac{R^2}{2}+\sum_{s=1}^{t} \frac{R^2}{2s}\right)~~~~({\rm by~ (C2^+)})\nonumber\\
		&\leq&\frac{RL}{c\sum_{s=1}^{t}\frac{1}{\sqrt{s}}}\left(\frac{1}{2}+\sum_{s=1}^{t} \frac{1}{2s}\right)\nonumber\\
		&\leq& \frac{2RL+RL\log t}{4c(\sqrt{t+1}-1)}, \label{eq12}
	\end{eqnarray}
	where \eqref{eq12} holds since $\sum_{s=1}^{t}\frac{1}{s} \le 1 + \log t$, and $\sum_{s=1}^{t} \frac{1}{\sqrt{s}} \geq 2(\sqrt{t+1}-1)$.
	
	Finally, we prove case (4) of Theorem \ref{cor:joint_design}. By \eqref{Convergence}, we have
	\begin{eqnarray}
		f\left( \frac{\sum_{s=1}^t \beta_s x_s}{\sum_{s=1}^t \beta_s} \right) - f(x^*) 
		&\leq& \frac{1}{\sum_{s=1}^{t}\beta_s}
		\left(
		\frac{R^2}{2 \alpha_t}
		+ \sum_{s=1}^{t} \frac{\alpha_s \beta_s}{2} \Vert g_s \Vert^2
		\right)\nonumber\\
		&\leq&\frac{1}{ct}\left(\frac{RG_t t^{\frac{a}{2}}}{2}+\sum_{s=1}^{t} \frac{R}{2G_s s^{\frac{a}{2}}}\Vert g_s \Vert^2\right)~~~~({\rm by~ (C2^+)})\nonumber\\
		&\leq& \frac{1}{ct}\left(\frac{R\max\limits_{s=1,\cdots,t}\|g_s\| \sqrt{t}}{2}+\sum_{s=1}^{t} \frac{R}{2 \sqrt{s}}\Vert g_s \Vert\right)~~~~({\rm by~\eqref{G_s1}})\nonumber\\
		&\leq& \frac{3R}{2c\sqrt{t}}\cdot \max_{s=1,\cdots,t}\|g_s\|, \label{eq11}
	\end{eqnarray}
	where \eqref{eq11} holds since $\sum_{s=1}^{t} \frac{1}{\sqrt{s}} \le 2\sqrt{t}$. The proof is complete.
\end{proof}

\begin{remark}
	Theorem \ref{cor:joint_design} shows that, through a joint design of the subgradient step size $\alpha_s$, the gliding step size $\beta_s$, and the ergodic weight $w_s$, condition (C1) in Theorem \ref{thm3} is no longer necessary as it is automatically satisfied. As a consequence, for classical choices of the subgradient step size $\alpha_s$ in PSG, the subgradient gliding method guarantees convergence for almost arbitrary choices of the gliding step size $\beta_s$, provided that (C2$^+$) holds. This substantially increases the flexibility of step size selection in practical implementations.
\end{remark}

\begin{remark}
	In Cases (2) and (4) of Theorem \ref{cor:joint_design}, the ergodic weights can likewise be chosen as $w_s = \beta_s w_s^{(k)}$, where $w_s^{(k)}$ is defined as in \eqref{weight}. Under this choice, condition (C1) in Theorem \ref{thm3} is also automatically satisfied, and convergence results analogous to \eqref{WeakConvergence1} and \eqref{WeakConvergence3} in Theorem \ref{thm3} can be obtained.
\end{remark}

\begin{remark}\label{r5}
	Even when $\|g_s\|$ is unbounded (i.e., Lipschitz assumption is violated), convergence of subgradient gliding method can still be assured by case (4) of Theorem \ref{thm3} and Theorem \ref{cor:joint_design}, as long as the growth rate of $\|g_s\|$ during the iteration strictly stays  within $\mathcal{O}(\sqrt{s})$.
\end{remark}


\section{Subgradient Gliding Method for Strongly Convex Function}
When the objective function $f$ is strongly convex, incorporating strong convexity information into the step size of PSG enables the method to achieve the optimal ergodic convergence rate of $\mathcal{O}(1/t)$. For the subgradient gliding method, the optimal $\mathcal{O}(1/t)$ ergodic convergence rate can be attained through a straightforward joint design of the subgradient step size $\alpha_s$ and gliding step size $\beta_s$. We formally present this result in the following theorem:


\begin{theorem}\label{thm4} 
	Suppose that the following conditions on the subgradient step size and gliding step size hold when the objective function $f$ is $\mu$-strongly convex:
	\begin{description}
		\item[(C3)] The step sizes satisfy $\alpha_s \beta_s = \dfrac{2}{\mu (s+1)}$ for all $s$;
		\item[(C2$^+$)] $0 < c \leq \beta_s <1$ holds for all $s$ where c is an arbitrary positive constant.
	\end{description}
	If the initial point $x_1 \in \operatorname{int}(\mathcal{X})$, then all iterates $x_s \in \operatorname{int}(\mathcal{X})$, and the subgradient gliding method achieves the following optimal ergodic convergence rate: \footnote{Same convergence result for $\min_{s=1,\ldots,t} f(x_s)$ can also be obtained based on similar analysis.}
	\begin{equation}\label{conver2}
		f\left(\sum_{s=1}^t \frac{2s}{t(t+1)} x_s\right) - f(x^*)
		\le
		\frac{2 \max\limits_{s=1,\dots,t} \|g_s\|^2}{c \mu (t+1)}.
	\end{equation}
\end{theorem}

\begin{proof}
	Since objective function $f$ is $\mu$-strongly convex, we have
	\begin{equation*}
		f(x_s)-f(x^*)\leq g_s^T(x_s-x^*) - \frac{\mu}{2}\|x_s-x^*\|^2. 
	\end{equation*}
	Applying the same derivation as in the proof of \eqref{ineq2} for Theorem \ref{thm2}, we obtain
	\begin{eqnarray}
		f(x_s)-f(x^*)&\leq& \left(\frac{1}{2\alpha_s\beta_s}-\frac{\mu}{2}\right)\Vert x_s-x^* \Vert^2-\frac{1}{2\alpha_s\beta_s}\Vert x_{s+1}-x^* \Vert^2+\frac{\alpha_s}{2}\Vert g_s \Vert^2 \nonumber \\
		&=& \frac{\mu(s-1)}{4}\Vert x_s-x^* \Vert^2-\frac{\mu(s+1)}{4}\Vert x_{s+1}-x^* \Vert^2 + \frac{1}{\mu\beta_s(s+1)}\Vert g_s \Vert^2. \nonumber
	\end{eqnarray}
	Hence, by $0 < c \leq \beta_s <1$,
	\begin{equation}\label{ineq6}
		s\left(f(x_s)-f(x^*)\right)\leq \frac{\mu}{4}\left((s-1)s\Vert x_s-x^* \Vert^2 - s(s+1)\Vert x_{s+1}-x^* \Vert^2\right) + \frac{1}{c\mu}\Vert g_s \Vert^2. 
	\end{equation}
	Summing the resulting inequality over $s=1$ to $s=t$, we have
	\begin{eqnarray}
		f\left(\sum_{s=1}^t\frac{2s}{t(t+1)}x_s\right)-f(x^*)&\leq& \sum_{s=1}^{t} \frac{2s}{t(t+1)}\left(f(x_s)-f(x^*)\right) ~~~~({\rm Jensen's~ inequality})\nonumber \\
		&\leq&  \frac{2}{t(t+1)}\sum_{s=1}^t \frac{1}{c\mu}\Vert g_s \Vert^2~~~~({\rm by~ \eqref{ineq6}})\nonumber\\
		&\leq& \frac{2\max\limits_{s=1,\dots,t}{\|g_s\|}^2}{c\mu(t+1)}.\nonumber
	\end{eqnarray}
	The proof is complete.
\end{proof}

	\begin{remark}
		Theorem \ref{thm4} explicitly provides a joint design criterion (i.e., (C3) and (C2$^+$)) for the subgradient step size $\alpha_s$ and gliding step size $\beta_s$ in the strongly convex setting. A simple admissible choice is to set $\beta_s\equiv \beta, ~s=1,\cdots,t$ for any fixed constant $\beta \in (0,1)$.
	\end{remark}
	
	\begin{corollary}\label{PSG_special2}
		When $\beta = 1$, the subgradient gliding method reduces to the classical PSG for strongly convex objectives, in which case the interior-point property of the iterates is no longer guaranteed.
	\end{corollary}
	
	Corollary \ref{PSG_special2} shows that the classical PSG for strongly convex objectives can also be viewed as a special case of the subgradient gliding framework. However, the conditions (C3) and (C2$^+$) in Theorem \ref{thm4} allow for significantly more flexible designs of the subgradient step size $\alpha_s$ and the gliding step size $\beta_s$. The practical benefits brought by this flexibility are further demonstrated in Section 6.3.
	
	\begin{remark}
		Similar to the case of convex objective functions, Theorem \ref{thm4} also demonstrates that under strong convexity, even when $\|g_s\|$ is unbounded (i.e., Lipschitz assumption is violated), convergence of the subgradient gliding method can still be guaranteed, provided the growth rate of $\|g_s\|$ throughout the iteration remains strictly within $\mathcal{O}(\sqrt{s})$.
	\end{remark}
	
\section{Stochastic Subgradient Gliding Method}
In this section, we discuss the stochastic subgradient gliding method. In practical scenarios, one often cannot obtain exact subgradient information; instead, we have access to a random vector. Specifically, for a given $x \in \mathcal{X}$, the stochastic first-order oracle returns an estimate $\tilde{g}(x)$ satisfying $\mathbb{E}(\tilde{g}(x) \vert x) \in \partial f(x)$. This setting corresponds to two common situations: (1) computational errors introduce random noise into the first-order information; (2) when $f(x) = \frac{1}{m}\sum_{i=1}^{m} f_i(x)$, computing the exact subgradient may be computationally expensive. In this case, by uniformly sampling some indexes from $\{1,\dots,m\}$ and computing only the subgradient of the sampled components, one can significantly reduce the computational cost while ensuring that the sample subgradient is an unbiased estimate of the true subgradient. The iteration schemes of stochastic subgradient gliding method becomes
\begin{equation*}
	\left\{	\begin{aligned}
		y_{s+1}&=x_s-\alpha_s \tilde{g}_s, ~\text{where } \mathbb{E}(\tilde{g}_s \vert x_s)\in\partial f(x_s),\\
		z_{s+1}&=  \argmin_{x\in\mathcal{X}}\Vert x-y_{s+1} \Vert,\\
		x_{s+1}&= (1-\beta_s)x_s +\beta_s z_{s+1}.	\end{aligned} \right.
\end{equation*}
Corresponding stochastic subgradient gliding methods can also be established for both convex and strongly convex objectives. We now present the theoretical results for the stochastic subgradient gliding method in the convex setting.

\begin{theorem}\label{thm5}
	Suppose that the conditions (C1) and (C2$^+$) in Theorem \ref{thm3} on $\beta_s$ hold. If the initial point $x_1 \in \mathcal{X}$, then all iterates $x_s \in \text{int}(\mathcal{X})$ and
	
	(1) Stochastic subgradient gliding method \footnote{For stochastic setting, Lipschitz constant $L$ is defined such that $\mathbb{E}\|\tilde{g}(x)\|^2 \le L^2$ , as in \cite{bubeck2015convex}.} with constant subgradient step-size \eqref{size1} satisfies (where ergodic weight $w_s$ is constant)   
	\begin{equation}\label{WeakConvergence4}
		\mathbb{E}f\left(\frac{\sum_{s=1}^t x_s}{t}\right)-f(x^*)\leq \frac{RL}{c\sqrt{t}},
	\end{equation}
	
	(2) Stochastic subgradient gliding method with time-varying subgradient step-size \eqref{size2} satisfies (where ergodic weight $w_s = s^{\frac{k}{2}}$ for any fixed $k \geq -1$)
	\begin{equation}\label{WeakConvergence5}
		\mathbb{E}f\left( \frac{\sum_{s=1}^t s^{\frac{k}{2}} x_s}{\sum_{s=1}^t s^{\frac{k}{2}}} \right) - f(x^*)
		\leq \frac{t^{\frac{k+1}{2}} +\sum_{s=1}^{t} s^{\frac{k-1}{2}}}{2c\sum_{s=1}^{t} s^{\frac{k}{2}}} RL.
	\end{equation}
	
	(3) Stochastic subgradient gliding method with subgradient step-size family \eqref{size3} satisfies (where ergodic weight $w_s = s^{\frac{k}{2}}$ for any fixed $k \geq 0$)
	\begin{equation}\label{WeakConvergence6}
		\mathbb{E}f\left( \frac{\sum_{s=1}^t s^{\frac{k}{2}} x_s}{\sum_{s=1}^t s^{\frac{k}{2}}} \right) - f(x^*)
		\leq \frac{t^{\frac{k+1}{2}} +\sum_{s=1}^{t} s^{\frac{k-1}{2}}}{2c\sum_{s=1}^{t} s^{\frac{k}{2}}} R\cdot\mathbb{E}\max\limits_{s=1,\dots,t}{\|\tilde{g}_s\|}.
	\end{equation}
\end{theorem}

\begin{proof}
	First, by the definition of subgradient,
	\begin{equation*}
		f(x_s)-f(x^*) \leq \mathbb{E}(\tilde{g}_s \vert x_s)^T(x_s-x^*).
	\end{equation*}
	Hence, 
	\begin{eqnarray}
		\mathbb{E}f(x_s)-f(x^*) &\leq& \mathbb{E}\left(\mathbb{E}(\tilde{g}_s \vert x_s)^T(x_s-x^*)\right) \nonumber \\
		&=& \mathbb{E}\left(\tilde{g}_s^T(x_s-x^*)\right).\nonumber
	\end{eqnarray}
	Applying similar derivation in \eqref{ineq2}, we have
	\begin{equation*}
		\mathbb{E}f(x_s)-f(x^*) \leq \mathbb{E}\left(\frac{1}{2\alpha_s}\cdot\frac{1}{\beta_s}(\Vert x_s-x^* \Vert^2-\Vert x_{s+1}-x^* \Vert^2)+\frac{\alpha_s}{2}\Vert \tilde{g}_s \Vert^2\right).
	\end{equation*}
	Hence,
	\begin{eqnarray}
		&&\mathbb{E}f\left( \frac{\sum_{s=1}^t w_s x_s}{\sum_{s=1}^t w_s} \right) - f(x^*) \nonumber \\
		&\leq& \left(\frac{1}{\sum_{s=1}^t w_s}\right)\mathbb{E}\sum_{s=1}^t w_s(f(x_s)-f(x^*)) ~~~~({\rm Jensen's~ inequality})\nonumber \\
		&=& \left(\frac{1}{\sum_{s=1}^t w_s}\right)\sum_{s=1}^t w_s(\mathbb{E}f(x_s)-f(x^*))\nonumber \\
		&\leq&  \left(\frac{1}{\sum_{s=1}^t w_s}\right)\sum_{s=1}^t w_s\mathbb{E}\left(\frac{1}{2\alpha_s}\cdot\frac{1}{\beta_s}(\Vert x_s-x^* \Vert^2-\Vert x_{s+1}-x^* \Vert^2)+\frac{\alpha_s}{2}\Vert \tilde{g}_s \Vert^2\right) \nonumber \\
		&=& \left(\frac{1}{\sum_{s=1}^t w_s}\right)(\mathbb{E}\frac{w_1}{2\alpha_1\beta_1}\Vert x_1-x^* \Vert^2+\frac{1}{2}\mathbb{E}\sum_{s=2}^t\left(\frac{w_s}{ \alpha_s\beta_s }-\frac{w_{s-1}}{ \alpha_{s-1}\beta_{s-1}}\right)\Vert x_s-x^* \Vert^2\nonumber\\
		&& -\mathbb{E}\frac{w_t^{(k)}}{ 2\alpha_{t}\beta_t }\Vert x_{t+1}-x^* \Vert^2+\mathbb{E}\sum_{s=1}^t\frac{w_s\alpha_s}{2}\Vert \tilde{g}_s \Vert^2)\nonumber\\
		&\leq& \left(\frac{1}{\sum_{s=1}^t w_s}\right)\mathbb{E}\left(R^2 \left( \frac{w_1}{2\alpha_1\beta_1}+\frac{1}{2}\sum_{s=2}^t\left(\frac{w_s}{ \alpha_s\beta_s }-\frac{w_{s-1}}{ \alpha_{s-1}\beta_{s-1}}\right) \right)+\sum_{s=1}^t\frac{w_s\alpha_s}{2}\Vert \tilde{g}_s \Vert^2\right) \nonumber\\
		&=& \left(\frac{1}{\sum_{s=1}^t w_s}\right)\mathbb{E}\left(\frac{R^2 w_t}{2\alpha_t\beta_t}+\sum_{s=1}^t\frac{w_s\alpha_s}{2}\Vert \tilde{g}_s \Vert^2\right). \label{Convergence1}
	\end{eqnarray}
	We first prove case (1), where $w_s$ is constant. By \eqref{Convergence1},
	\begin{eqnarray}
		\mathbb{E}f\left(\frac{\sum_{s=1}^t x_s}{t}\right)-f(x^*) 
		&\leq& \frac{1}{t}\mathbb{E}\left(\frac{R^2}{2\alpha_t\beta_t}+\sum_{s=1}^{t}\frac{\alpha_s}{2}\Vert \tilde{g}_s \Vert^2\right)\nonumber\\
		&\leq&\frac{1}{t}\mathbb{E}\left(\frac{RL\sqrt{t}}{2\beta_t}+\frac{RL\sqrt{t}}{2}\right)~~~~({\rm by~ \eqref{size1}})\nonumber\\
		&\leq& \frac{RL}{c\sqrt{t}}.~~~~({\rm by~ 0 < c \leq \beta_s <1})\nonumber
	\end{eqnarray}
	We now prove case (2).
	\begin{eqnarray}
		\mathbb{E}f\left( \frac{\sum_{s=1}^t w_s^{(k)} x_s}{\sum_{s=1}^t w_s^{(k)}} \right) - f(x^*)
		&\leq& \frac{1}{\sum_{s=1}^{t} s^{\frac{k}{2}}}\mathbb{E}\left(\frac{R^2 t^{\frac{k}{2}}}{2\alpha_t\beta_t}+\sum_{s=1}^{t}\frac{s^{\frac{k}{2}}\alpha_s}{2}\Vert \tilde{g}_s \Vert^2\right)\nonumber\\
		&\leq&\frac{1}{\sum_{s=1}^{t} s^{\frac{k}{2}}}\mathbb{E}\left(\frac{RL t^{\frac{k+1}{2}}}{2\beta_t}+\sum_{s=1}^{t}\frac{RLs^{\frac{k-1}{2}}}{2}\right)~~~~({\rm by~ \eqref{size2}})\nonumber\\
		&\leq& \frac{t^{\frac{k+1}{2}} +\sum_{s=1}^{t} s^{\frac{k-1}{2}}}{2c\sum_{s=1}^{t} s^{\frac{k}{2}}} RL.~~~~({\rm by~ 0 < c \leq \beta_s <1})\nonumber
	\end{eqnarray}
	We conclude with the proof of case (3). By definition of $G_s$ in \eqref{G_s}, we also have the equivalent form of $G_s$ as follows:
	\begin{equation}\label{G_s2}
		G_s = \max\limits_{j=1,\dots,s}{\{\|\tilde{g}_j\|j^{\frac{1-a}{2}}\}}.
	\end{equation}
	Considering the family of time-varying step sizes \eqref{size3} and the ergodic weights \eqref{weight} and substituting them into \eqref{Convergence1}, for $k \geq 0$, we obtain
	\begin{eqnarray}
		&&\mathbb{E}f\left( \frac{\sum_{s=1}^t w_s^{(k)} x_s}{\sum_{s=1}^t w_s^{(k)}} \right) - f(x^*)\nonumber \\
		&\leq& \frac{\mathbb{E}\left(\frac{R^2 t^{\frac{k}{2}}}{2\alpha_t\beta_t}+\sum_{s=1}^t\frac{s^{\frac{k}{2}}\alpha_s}{2}\Vert \tilde{g}_s \Vert^2\right)}{\sum_{s=1}^t s^{\frac{k}{2}}} \nonumber \\
		&\leq& \frac{R}{2c}\frac{\mathbb{E}\left(t^{\frac{k}{2}}G_t t^{\frac{a}{2}}+\sum_{s=1}^t\frac{s^{\frac{k}{2}}\Vert \tilde{g}_s \Vert^2}{G_s s^{\frac{a}{2}}}\right)}{\sum_{s=1}^t s^{\frac{k}{2}}} ~~~~({\rm by~ \eqref{size3}~ and ~0 < c \leq \beta_s <1})\nonumber \\
		&\leq& \frac{R}{2c}\frac{\mathbb{E}\left(t^{\frac{k+1}{2}}\max\limits_{s=1,\cdots,t}{\|\tilde{g}_s\|}+\sum_{s=1}^t\Vert \tilde{g}_s \Vert s^{\frac{k-1}{2}}\right)}{\sum_{s=1}^t s^{\frac{k}{2}}}\label{s3} \\
		&\leq&
		\frac{t^{\frac{k+1}{2}} +\sum_{s=1}^{t} s^{\frac{k-1}{2}}}{2c\sum_{s=1}^{t} s^{\frac{k}{2}}} R\cdot \mathbb{E}\max\limits_{s=1,\dots,t}{\|\tilde{g}_s\|}, \nonumber
	\end{eqnarray}
	where \eqref{s3} follows from definition \eqref{G_s}, \eqref{G_s2} and $0 \leq a \leq 1$ when $k \geq 0$. The proof is complete.
\end{proof}

\begin{remark}
	For $k > -1$ in case (2) and in case (3), the optimal ergodic convergence rate of $\mathcal{O}(1/\sqrt{t})$ is attained. For the stochastic subgradient gliding method, employing the same step size design principles (C1) and (C2$^+$) established in Theorem \ref{thm3} for the exact setting also achieves the optimal ergodic convergence rates in expectation for the respective step size choices.
\end{remark}

\begin{theorem}[A practical joint design that eliminates (C1)]\label{cor:joint_design1}
	Suppose that the gliding step size $\beta_s$ satisfies condition (C2$^+$).
	If the initial point $x_1 \in \operatorname{int}(\mathcal{X})$, then all iterates $x_s \in \operatorname{int}(\mathcal{X})$, and
	
	(1) For the stochastic subgradient gliding method with the constant subgradient step size \eqref{size1} and  ergodic weights $w_s = \beta_s$, it holds that
	\begin{equation}\label{WeakConvergence5pd}
		\mathbb{E}f\left( \frac{\sum_{s=1}^t \beta_s x_s}{\sum_{s=1}^t \beta_s} \right) - f(x^*)
		\leq \frac{RL}{c\sqrt{t}}.
	\end{equation}
	
	(2) For the stochastic subgradient gliding method with the time-varying subgradient step size \eqref{size2} and ergodic weights $w_s = \beta_s$, it holds that
	\begin{equation}\label{WeakConvergence6pd}
		\mathbb{E}f\left( \frac{\sum_{s=1}^t \beta_s x_s}{\sum_{s=1}^t \beta_s} \right) - f(x^*)
		\leq \frac{3RL}{2c\sqrt{t}}.
	\end{equation}
	
	(3) For the stochastic subgradient gliding method with the subgradient step size family \eqref{size3} and ergodic weights $w_s = \beta_s$, it holds that
	\begin{equation}\label{WeakConvergence7pd}
		\mathbb{E}f\left( \frac{\sum_{s=1}^t \beta_s x_s}{\sum_{s=1}^t \beta_s} \right) - f(x^*)
		\leq \frac{3R}{2c\sqrt{t}}\cdot \mathbb{E}\max_{s=1,\cdots,t}\|g_s\|.
	\end{equation}
\end{theorem}
\begin{proof}
	Similar to the proof of Theorem \ref{cor:joint_design}, the joint designs in case (1), (2) and (3) automatically safisfy (C1) in Theorem \ref{thm3}. Hence, \eqref{Convergence1} holds for case (1), (2) and (3).
	
	We first prove case (1). By \eqref{Convergence1},
	\begin{eqnarray}
		\mathbb{E}f\left( \frac{\sum_{s=1}^t \beta_s x_s}{\sum_{s=1}^t \beta_s} \right) - f(x^*)
		&\leq& \frac{1}{\sum_{s=1}^t \beta_s}\mathbb{E}\left(\frac{R^2}{2\alpha_t}+\sum_{s=1}^{t}\frac{\alpha_s\beta_s}{2}\Vert \tilde{g}_s \Vert^2\right)\nonumber\\
		&\leq&\frac{1}{ct} \left(\frac{RL\sqrt{t}}{2}+\frac{RL\sqrt{t}}{2}\right)~~~~({\rm by~ (C2^+)})\nonumber\\
		&\leq& \frac{RL}{c\sqrt{t}}.\nonumber
	\end{eqnarray}
	
	We proceed to prove case (2). By \eqref{Convergence1}, we have
	\begin{eqnarray}
		\mathbb{E}f\left( \frac{\sum_{s=1}^t \beta_s x_s}{\sum_{s=1}^t \beta_s} \right) - f(x^*) 
		&\leq& \frac{1}{\sum_{s=1}^t \beta_s}\mathbb{E}\left(\frac{R^2}{2\alpha_t}+\sum_{s=1}^{t}\frac{\alpha_s\beta_s}{2}\Vert \tilde{g}_s \Vert^2\right)\nonumber\\
		&\leq& \frac{1}{\sum_{s=1}^{t}\beta_s}
		\left(
		\frac{R^2}{2 \alpha_t}
		+ \sum_{s=1}^{t} \frac{\alpha_s\beta_s}{2} L^2
		\right)\nonumber\\
		&\leq&\frac{1}{ct}\left(\frac{RL\sqrt{t}}{2}+\sum_{s=1}^{t} \frac{RL}{2\sqrt{s}}\right)~~~~({\rm by~ (C2^+)})\nonumber\\
		&\leq& \frac{3RL}{2c\sqrt{t}},\label{eq15}
	\end{eqnarray}
	where \eqref{eq15} holds since $\sum_{s=1}^{t} \frac{1}{\sqrt{s}} \le 2\sqrt{t}$.
	
	Finally, we prove case (3). By \eqref{Convergence1}, we have
	\begin{eqnarray}
		\mathbb{E}f\left( \frac{\sum_{s=1}^t \beta_s x_s}{\sum_{s=1}^t \beta_s} \right) - f(x^*) 
		&\leq& \frac{1}{\sum_{s=1}^t \beta_s}\mathbb{E}\left(\frac{R^2}{2\alpha_t}+\sum_{s=1}^{t}\frac{\alpha_s\beta_s}{2}\Vert \tilde{g}_s \Vert^2\right)\nonumber\\
		&\leq& \frac{1}{ct}\mathbb{E}\left(\frac{R^2}{2\alpha_t}+\sum_{s=1}^{t}\frac{\alpha_s}{2}\Vert \tilde{g}_s \Vert^2\right)~~~~({\rm by~ (C2^+)})\nonumber\\
		&=& \frac{1}{ct}\mathbb{E}\left(\frac{RG_t t^{\frac{a}{2}}}{2}+\sum_{s=1}^{t}\frac{R}{2G_s s^{\frac{a}{2}}}\Vert \tilde{g}_s \Vert^2\right)\nonumber\\
		&\leq&
		\frac{1}{ct}\mathbb{E}\left(\frac{R\max\limits_{s=1,\cdots,t}{\|\tilde{g}_s\|} t^{\frac{1}{2}}}{2}+\sum_{s=1}^{t}\frac{R}{2 s^{\frac{1}{2}}}\Vert \tilde{g}_s \Vert\right)~~~~({\rm by~ \eqref{G_s2}})\nonumber\\
		&\leq&
		\frac{R}{ct}\left(\frac{\sqrt{t}}{2}+\sum_{s=1}^{t}\frac{1}{2\sqrt{s}}\right)\mathbb{E}\max\limits_{s=1,\cdots,t}{\|\tilde{g}_s\|}\nonumber\\
		&\leq& \frac{3R}{2c\sqrt{t}}\cdot \mathbb{E}\max_{s=1,\cdots,t}\|g_s\|, \label{eq19}
	\end{eqnarray}
	where \eqref{eq19} holds since $\sum_{s=1}^{t} \frac{1}{\sqrt{s}} \le 2\sqrt{t}$. The proof is complete.
\end{proof}

Theorem \ref{cor:joint_design1} shows that, through a joint design of the step sizes and ergodic weights, the stochastic subgradient gliding method can also guarantee optimal convergence rate in expectation without relying on condition (C1). The following theorem establishes the convergence rate of the stochastic subgradient gliding method in the strongly convex setting.

\begin{theorem}\label{thm6}
	Suppose that the conditions (C3) and (C2$^+$) in Theorem \ref{thm4} on the subgradient step size and gliding step size hold when the objective function $f$ is $\mu$-strongly convex. If the initial point $x_1 \in \text
	{int}(\mathcal{X})$, then all iterates of $x_s \in \text{int}(\mathcal{X})$ and stochastic subgradient gliding method achieves the following optimal ergodic convergence rate:
	\begin{equation}\label{conver3}
		\mathbb{E}f\left(\sum_{s=1}^t\frac{2s}{t(t+1)}x_s\right)-f(x^*)\leq
		\frac{2\cdot\mathbb{E}\max\limits_{s=1,\dots,t}{\|\tilde{g}_s\|}^2}{c\mu(t+1)}.
	\end{equation}
\end{theorem}

\begin{proof}
	Since objective function $f$ is $\mu$-strongly convex, we have
	\begin{equation*}
		f(x_s)-f(x^*)\leq \mathbb{E}(\tilde{g}_s \vert x_s)^T(x_s-x^*) - \frac{\mu}{2}\|x_s-x^*\|^2. 
	\end{equation*}
	Hence,
	\begin{eqnarray}
		\mathbb{E}f(x_s)-f(x^*)&\leq& \mathbb{E}\left(\mathbb{E}(\tilde{g}_s \vert x_s)^T(x_s-x^*) - \frac{\mu}{2}\|x_s-x^*\|^2\right) \nonumber \\
		&=& \mathbb{E}\left(\tilde{g}_s^T(x_s-x^*) - \frac{\mu}{2}\|x_s-x^*\|^2\right). \nonumber
	\end{eqnarray}
	Applying the same derivation as in the proof of \eqref{ineq2} for Theorem \ref{thm2}, we obtain
	\begin{eqnarray}
		\mathbb{E}f(x_s)-f(x^*)&\leq& \mathbb{E}\left(\left(\frac{1}{2\alpha_s\beta_s}-\frac{\mu}{2}\right)\Vert x_s-x^* \Vert^2-\frac{1}{2\alpha_s\beta_s}\Vert x_{s+1}-x^* \Vert^2+\frac{\alpha_s}{2}\Vert \tilde{g}_s \Vert^2\right) \nonumber \\
		&=& \mathbb{E}\left(\frac{\mu(s-1)}{4}\Vert x_s-x^* \Vert^2-\frac{\mu(s+1)}{4}\Vert x_{s+1}-x^* \Vert^2 + \frac{1}{\mu\beta_s(s+1)}\Vert \tilde{g}_s \Vert^2\right). \nonumber
	\end{eqnarray}
	Hence, by $0 < c \leq \beta_s <1$,
	\begin{equation}\label{ineq7}
		s\mathbb{E}\left(f(x_s)-f(x^*)\right)\leq \mathbb{E}\left(\frac{\mu}{4}\left((s-1)s\Vert x_s-x^* \Vert^2 - s(s+1)\Vert x_{s+1}-x^* \Vert^2\right) + \frac{1}{c\mu}\Vert \tilde{g}_s \Vert^2\right). 
	\end{equation}
	Summing the resulting inequality over $s=1$ to $s=t$, we have
	\begin{eqnarray}
		\mathbb{E}f\left(\sum_{s=1}^t\frac{2s}{t(t+1)}x_s\right)-f(x^*)&\leq& \mathbb{E}\sum_{s=1}^{t} \frac{2s}{t(t+1)}\left(f(x_s)-f(x^*)\right) ~~~~({\rm Jensen's~ inequality})\nonumber \\
		&=& \sum_{s=1}^{t} \frac{2s}{t(t+1)}\mathbb{E}\left(f(x_s)-f(x^*)\right) \nonumber \\
		&\leq&  \frac{2}{t(t+1)}\mathbb{E}\sum_{s=1}^t \frac{1}{c\mu}\Vert \tilde{g}_s \Vert^2~~~~({\rm by~ \eqref{ineq7}})\nonumber\\
		&\leq& \frac{2\cdot\mathbb{E}\max\limits_{s=1,\dots,t}{\|\tilde{g}_s\|}^2}{c\mu(t+1)}.\nonumber
	\end{eqnarray}
	The proof is complete.
\end{proof}	

\begin{remark}
	In the strongly convex setting, the stochastic subgradient gliding method, employing the same step size design principles (C3) and (C2$^+$) in Theorem \ref{thm4} for the exact scenario, also achieves the optimal ergodic convergence rate of $\mathcal{O}(1/t)$ in expectation.
\end{remark}

\section{Experiments}
In this section, we present numerical results on the nonsmooth convex optimization examples from Section 1.2 and 2, in which subgradients do not exist at the boundary and PSG fails with arbitrarily high probability. The experiments demonstrate that, compared to the classical PSG, the subgradient gliding method not only achieves consistent success on such problems but can also yield stronger practical performance gains through appropriate selection of the gliding step size.

\subsection{Experiments on Example \ref{e1}}
In this subsection, we test the performance of the classical PSG and the subgradient gliding method on Example \ref{e1}, i.e.,
\begin{equation}
	\begin{alignedat}{2}
		&\min_{x = (x(1), x(2))} &&-\sqrt{r - k_1x(1)^2-k_2x(2)^2}\\
		&~~~~~~\ST&&k_1x(1)^2+k_2x(2)^2\le r,
		\label{m1}
	\end{alignedat}
\end{equation}
where $r, k_1, k_2 > 0$. This objective function does not admit subgradients at any boundary point. Here, we employ time-varying step sizes \eqref{Nesterov} and \eqref{size3} for the classical PSG and the subgradient gliding method, respectively. For simplicity, the gliding step size $\beta_s$ of the subgradient gliding method is set to a constant in $(0,1)$. Since the optimal solution of this objective function is at $x^*=(0,0)$, we take $R = \max\{\sqrt{\frac{r}{k_1}}, \sqrt{\frac{r}{k_2}}\}$ according to the definition of $R$, i.e., the major axis length of the ellipse. We have shown in Theorem \ref{thm2.1} that, as $k_2/k_1$ increases, the classical PSG fails with arbitrarily high probability. Here, we report the success rates of the classical PSG and the subgradient gliding method on Example \ref{e1} as $k_2/k_1$ varies. 

A successful trial is defined as follows: starting from an initial point, the algorithm runs smoothly for 100 iterations without encountering points where the subgradient does not exist, and satisfies $\min_{s \in \{1,\dots,100\}} f_s - f^* \le 1 \times 10^{-9}$. We uniformly generate 1000 sets of initial points within the elliptical region $k_1x(1)^2+k_2x(2)^2< r$. For the classical PSG and the subgradient gliding method, using the time-varying step sizes \eqref{Nesterov} and \eqref{size3} respectively, the success rates under different parameter settings are shown in Table \ref{tab:success_rate1}.

\vspace{-4mm}
\begin{table}[htbp]
	\centering
	\caption{Success rate comparison between PSG and subgradient gliding method (SGM) with different time-varying step sizes ($r=100$)}
	\begin{adjustbox}{width=1\textwidth}
		\begin{tabular}{ccccccc}
			\toprule 
			\multirow{2}{*}{\textbf{Method}} & \multicolumn{5}{c}{\textbf{Success rate on step size \eqref{Nesterov}}} & \\
			\cmidrule{2-7} 
			& $k_1=2, k_2=5$ & $k_1=2, k_2=7$ & $k_1=2, k_2=10$ & $k_1=2, k_2=15$ & $k_1=2, k_2=20$ & \\
			\midrule 
			\textbf{PSG} & $78.1\%$& $37\%$& $10.9\%$& $0.3\%$& $0\%$&\\
			\cmidrule{1-7}
			\textbf{SGM} & $100\%$ & $100\%$ & $100\%$ & $100\%$ & $100\%$ &\\
			\cmidrule[1pt]{1-7} 
			\multirow{2}{*}{\textbf{Method}} & \multicolumn{5}{c}{\textbf{Success rate on step size \eqref{size3}}} & \\
			\cmidrule{2-7} 
			& $k_1=2, k_2=5$ & $k_1=2, k_2=7$ & $k_1=2, k_2=10$ & $k_1=2, k_2=15$ & $k_1=2, k_2=20$ & \\
			\midrule 
			\textbf{PSG} & $78.4\%$& $38.9\%$& $15.1\%$& $2.1\%$& $0.8\%$&\\
			\cmidrule{1-7}
			\textbf{SGM} & $100\%$ & $100\%$ & $100\%$ & $100\%$ & $100\%$ &\\
			\bottomrule 
		\end{tabular}
	\end{adjustbox}
	\label{tab:success_rate1}
\end{table}
\vspace{-2mm}

From Table \ref{tab:success_rate1}, it is clearly observed that regardless of the step size used, the success rate of classical PSG decreases significantly to zero as the ratio $k_2/k_1$ increases. In contrast, the subgradient gliding method consistently achieves stable success rate across all these different initial points. This provides further experimental support for the discussion in Theorem \ref{thm2.1}. Furthermore, we visualize in Figure \ref{fig2} the initial points where the classical PSG iteration succeeds or fails under different objective function settings. In Figure \ref{fig2}, red circular points represent the initial points where PSG succeeds, while blue triangular points represent those where PSG fails. The evolution pattern of the PSG failure region in Figure \ref{fig2} as $k_2/k_1$ increases is also consistent with the theoritical results analyzed in Figure \ref{fig7}.

\begin{figure}[htbp]
	\centering
	\includegraphics[width=\linewidth]{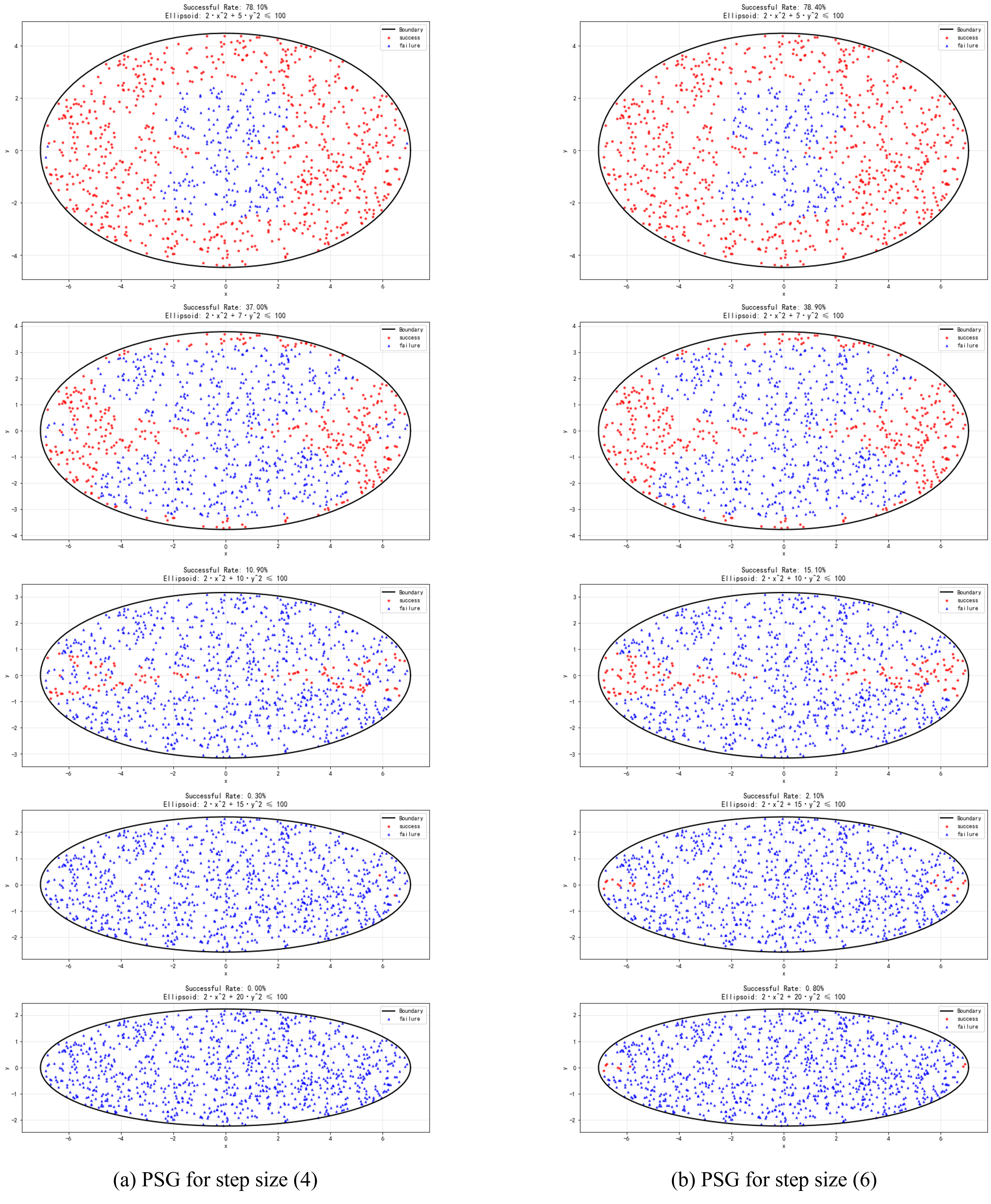}
	\vspace{2mm}
	\caption{Visualization of initial points where the classical PSG iteration succeeds or fails under different objective function settings.}
	\label{fig2}
\end{figure}

\begin{figure}[htbp]
	\centering
	\includegraphics[width=\linewidth]{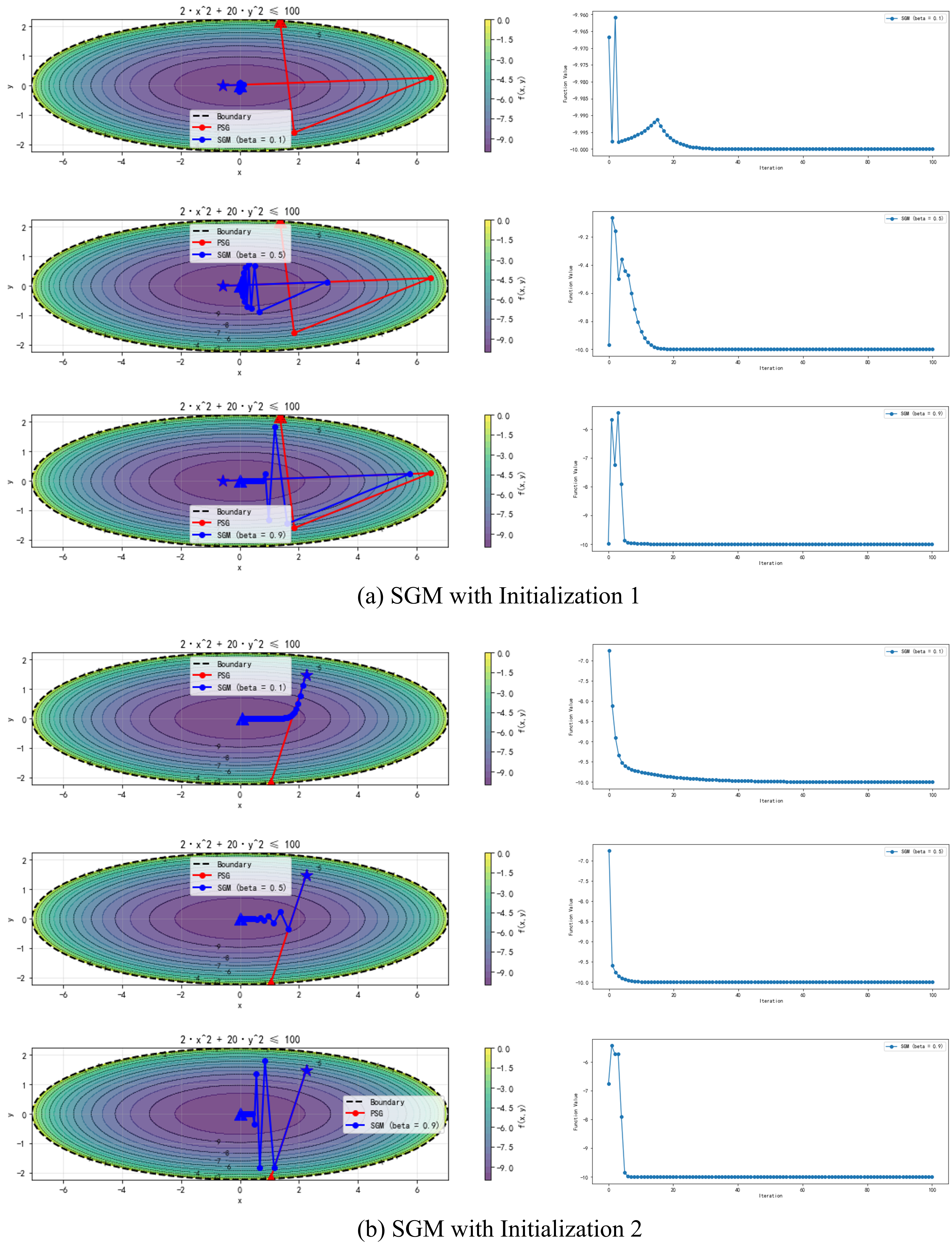}
	\vspace{2mm}
	\caption{The impact of different choices of the gliding step size $\beta_s\equiv \beta$, where $\beta=0.1, 0.5, 0.9$ in the subgradient gliding method.}
	\label{fig3}
\end{figure}

We further demonstrate the impact of different choices of the gliding step size on the actual convergence rate. Here, we adopt the time‑varying step size \eqref{size3} as the subgradient step size for the subgradient gliding method. For simplicity, only the cases with constant gliding step sizes are illustrated. In fact, the gliding step size can be designed more flexibly as long as it satisfies the conditions (C2$^+$) according to Theorem \ref{cor:joint_design}. The experimental results are presented in Figure \ref{fig3}.

Figure \ref{fig3} illustrates the iteration paths of the classical PSG and the subgradient gliding method under different initial points. In the figure, the star‑shaped points denote the initial points for both the classical PSG and the subgradient gliding method, while the triangular points indicate the iteration endpoints of the respective methods. It can be observed that the classical PSG method easily reach the boundary, causing the algorithm to halt, whereas the subgradient gliding method consistently converges stably and successfully. Furthermore, the selection of the gliding step size influences the actual convergence rate of the algorithm. In example (a) of Figure \ref{fig3}, choosing a small gliding step size $\beta = 0.1$ yields faster convergence; conversely, in example (b), larger gliding step sizes $\beta = 0.5$ and $0.9$ lead to faster convergence. This further demonstrates that, in practical problems, a more flexible design of the gliding step size according to (C2$^+$) in Theorem \ref{thm3} and Theorem \ref{cor:joint_design} can bring about more substantial practical benefits while maintaining the theoretical guarantees.

\subsection{Experiments on Example \ref{e2}}
In this section, we evaluate the performance of the classical PSG and the subgradient gliding method on Example \ref{e2}, i.e.,
\begin{equation}
	\begin{alignedat}{2}
		&\min_{x = (x(1), x(2))} && f(x) = 
		\begin{cases}
			0 & \text{if } (x(1), x(2)) = (0,0), \\
			\frac{x(1)^2 + x(2)^2}{x(1)} & \text{if } x(1) > 0, \\
			\infty & \text{otherwise}.
		\end{cases}\\
		&~~~~~~\ST&&0 \le x(1) \le 1, -1 \le x(2) \le 1.
		\label{m2}
	\end{alignedat}
\end{equation}
The objective function of this optimization problem does not satisfy the Lipschitz condition on any level set, and the subgradient does not exist on the boundary where $x(1)=0$, $-1 \le x(2) \le 1$, $x(2) \neq 0$. Here, we also employ the time-varying step sizes \eqref{Nesterov} and \eqref{size3} for both the classical PSG and the subgradient gliding method. For simplicity, the gliding step size for the subgradient gliding method is set as a constant within $(0,1)$. Since the optimal solution of this objective function is at $(0,0)$, we take $R = \sqrt{2}$ according to the definition of $R$. We randomly generated 5 sets of initial points within the region $0 < x(1) < 1, -1 < x(2) < 1$ to test the algorithm performance, and the experimental results are presented in Table \ref{tab:success_rate2}.

\begin{table}[htbp]
	\centering
	\caption{Experiment comparison between PSG and subgradient gliding method (SGM) with different time-varying step sizes}
	\begin{adjustbox}{width=1\textwidth}
		\begin{tabular}{ccccccc}
			\toprule 
			\multicolumn{6}{c}{\textbf{Experiment on step size \eqref{Nesterov}}} & \\
			\cmidrule{1-7} 
			\textbf{Test} & \textbf{Method} & \textbf{Preset IterNum} & \textbf{Actual IterNum} & $\min_{s} f(x_s)$ & \textbf{Succeed or fail}&\\
			\midrule 
			\multirow{2}{*}{\textbf{Case 1}} & \textbf{PSG}& 10000& 2& 0.9005& fail&\\
			\cmidrule{2-7}
			& \textbf{SGM} & 10000 & 10000 & 0.0014 & succeed &\\
			\midrule 
			\multirow{2}{*}{\textbf{Case 2}} & \textbf{PSG}& 10000& 7& 0.1177& fail&\\
			\cmidrule{2-7}
			& \textbf{SGM} & 10000 & 10000 & 0.0014 & succeed &\\
			\midrule 
			\multirow{2}{*}{\textbf{Case 3}} & \textbf{PSG}& 100& 7& 0.0753& fail&\\
			\cmidrule{2-7}
			& \textbf{SGM} & 100 & 100 & 0.0036 & succeed &\\
			\midrule 
			\multirow{2}{*}{\textbf{Case 4}} & \textbf{PSG}& 1000& 5& 0.6459& fail&\\
			\cmidrule{2-7}
			& \textbf{SGM} & 1000 & 1000 & 0.0038 & succeed &\\
			\midrule 
			\multirow{2}{*}{\textbf{Case 5}} & \textbf{PSG}& 1000& 5& 0.5502& fail&\\
			\cmidrule{2-7}
			& \textbf{SGM} & 1000 & 1000 & 0.0035 & succeed &\\
			
			\cmidrule[1pt]{1-7} 
			\multicolumn{6}{c}{\textbf{Experiment on step size \eqref{size3}}} & \\
			\cmidrule{1-7} 
			\textbf{Test} & \textbf{Method} & \textbf{Preset IterNum} & \textbf{Actual IterNum} & $\min_{s} f(x_s)$ & \textbf{Succeed or fail}&\\
			\midrule 
			\multirow{2}{*}{\textbf{Case 1}} & \textbf{PSG}& 10000& 2& 0.9005& fail&\\
			\cmidrule{2-7}
			& \textbf{SGM} & 10000 & 10000 & $1.2801\cdot10^{-7}$ & succeed &\\
			\midrule 
			\multirow{2}{*}{\textbf{Case 2}} & \textbf{PSG}& 10000& 5& 0.2933& fail&\\
			\cmidrule{2-7}
			& \textbf{SGM} & 10000 & 10000 & $4.3889\cdot10^{-7}$ & succeed &\\
			\midrule 
			\multirow{2}{*}{\textbf{Case 3}} & \textbf{PSG}& 100& 24& 0.0471& fail&\\
			\cmidrule{2-7}
			& \textbf{SGM} & 100 & 100 & $5.7492\cdot10^{-4}$ & succeed &\\
			\midrule 
			\multirow{2}{*}{\textbf{Case 4}} & \textbf{PSG}& 1000& 42& 0.0082& fail&\\
			\cmidrule{2-7}
			& \textbf{SGM} & 1000 & 1000 & $2.1517\cdot10^{-5}$ & succeed &\\
			\midrule 
			\multirow{2}{*}{\textbf{Case 5}} & \textbf{PSG}& 1000& 61& 0.0141& fail&\\
			\cmidrule{2-7}
			& \textbf{SGM} & 1000 & 1000 & $6.4942\cdot10^{-6}$ & succeed &\\
			\bottomrule 
		\end{tabular}
	\end{adjustbox}
	\label{tab:success_rate2}
\end{table}

In Table \ref{tab:success_rate2}, “Preset IterNum” denotes the preset number of iterations, while “Actual IterNum” represents the actual number of iterations performed by the algorithm (the algorithm terminates upon reaching the boundary where the subgradient is undefined). $\min_{s} f(x_s)$ records the best function value encountered during the iterations. A trial is considered successful if the algorithm runs without abnormal termination and satisfies $\min_s f_s - f^* \le 5\times10^{-3}$. As observed from Table \ref{tab:success_rate2}, regardless of the subgradient step size, the PSG method always terminates prematurely due to the absence of the subgradient during the iteration, failing to meet the desired accuracy. In contrast, the subgradient gliding method consistently achieves stable and successful convergence. Furthermore, by tuning the maximum iteration count, the subgradient gliding method can approximate the optimum with arbitrary precision. Additionally, comparing the results of time‑varying step sizes \eqref{Nesterov} and \eqref{size3} reveals that the choice of subgradient step size $\alpha_s$ influences the convergence rate of the subgradient gliding method. The time‑varying step size \eqref{size3} not only enjoys superior theoretical properties (see Theorem \ref{thm3} and Theorem \ref{cor:joint_design}) but also leads to significantly faster practical convergence, often by three to four orders of magnitude compared to step size \eqref{Nesterov}. We further illustrate the iteration trajectories of classical PSG and the subgradient gliding method using time‑varying step size \eqref{size3} in Figure \ref{fig4}.

\begin{figure}[htbp]
	\centering
	\includegraphics[width=\linewidth]{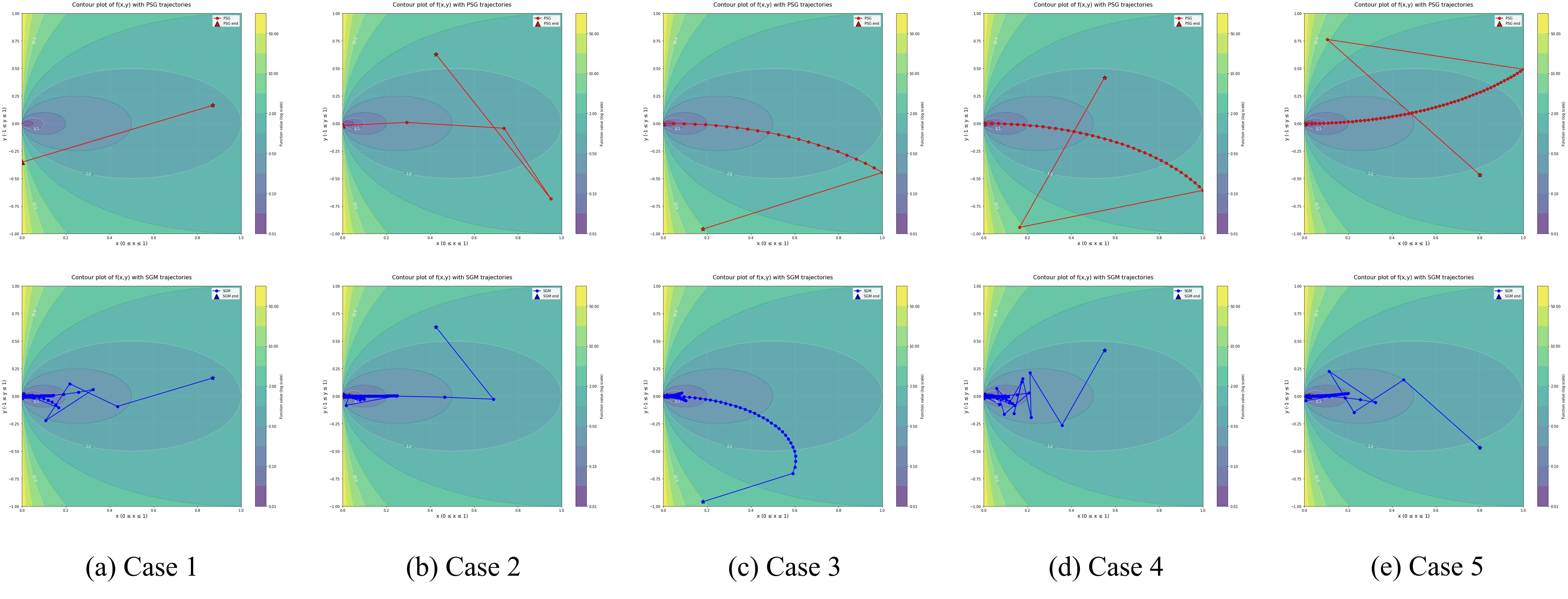}
	\vspace{-2mm}
	\caption{Iteration comparison between classic PSG and the subgradient gliding method (SGM).}
	\label{fig4}
\end{figure}

From Figure \ref{fig4}, it is clearly seen that the optimum of Example \ref{e2} lies on the boundary of the feasible region. This implies that once the algorithm approaches the optimal solution, it is likely to run out of the boundary from the left side, which prevents further iteration. This shows that for the classic PSG method, even if many iterations have been carried out, once it exits the boundary and the subgradient ceases to exist, the algorithm must stop and the accuracy remains fixed at that point. Therefore, it is difficult for PSG to achieve a high‑precision approximation of the optimum. In contrast, the subgradient gliding method can approach the optimum with arbitrary accuracy by increasing the number of iterations.

Another interesting aspect of this example, as noted in Remark \ref{r5}, is that the norm of the subgradient $\|g_s\|$ during the iteration is not bounded and tends to increase abruptly as the iteration proceeds. This occurs because the example does not satisfy the Lipschitz assumption on any level set, meaning that even when the algorithm is already very close to the optimum, iterates with large subgradient norms can still appear. In Figure \ref{fig5}, we further illustrate the variation of the function value and the subgradient norm during the iteration of the subgradient gliding method. This shows that the Lipschitz assumption is not an essential requirement for guaranteeing the convergence of subgradient methods. In fact, Theorem \ref{cor:joint_design} and Remark \ref{r5} have already indicated that the convergence of the subgradient gliding method does not require global Lipschitz continuity, but instead depends solely on the behavior of the subgradient norms along the iterate sequence. In Section 7, we provide a detailed convergence analysis of SGM for this challenging example.

\begin{figure}[htbp]
	\centering
	\includegraphics[width=\linewidth]{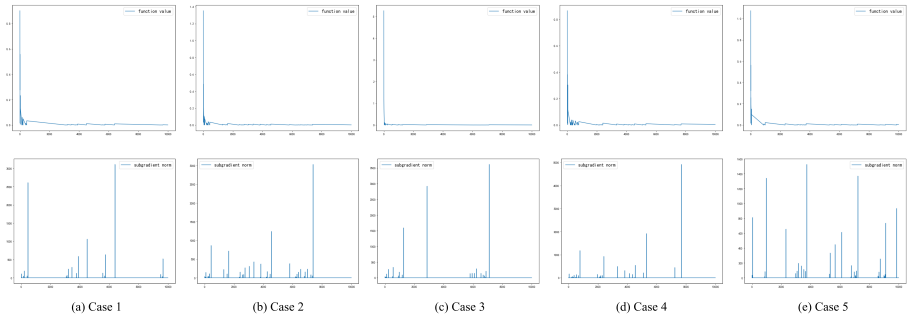}
	\vspace{-2mm}
	\caption{The variation of the function value and the subgradient norm during the iteration of the subgradient gliding method (10000 iterations).}
	\label{fig5}
\end{figure}

\subsection{Experiments on Example \ref{e3}}
In this section, we test the performance of the classical PSG and the subgradient gliding method on Example \ref{e3}, i.e.,
\begin{equation}
	\begin{alignedat}{2}
		&\min_{x} && ~ f(x)=\sum_{i=1}^{n} x(i) \log x(i)\\
		&~\ST&&0 \le x(i) \le B,~ i = 1,\dots n.
		\label{m3}
	\end{alignedat}
\end{equation}
For any $0 < \mu \le 1/B$, the objective function is $\mu$-strongly convex. In the step size \eqref{size4}, we choose the optimal strong convexity parameter, namely $\mu = 1/B$. As established in Theorem \ref{thm2.2}, when $B > 1$, the classical PSG fails with arbitrarily high probability as the problem dimension increases. Here, we consider different values of $B$ and examine how the success rates of the classical PSG and the subgradient gliding method evolve with the dimension $n$ on Example \ref{e3}. For simplicity, the gliding step size $\beta_s$ of the subgradient gliding method is set to a constant in $(0,1)$

A successful run is defined as one in which, starting from an initial point, the algorithm can proceed smoothly for 10 iterations without termination due to the absence of subgradient, while satisfying $\min_{s \in \{1,\dots,10\}} f_s - f^* \le 1 \times 10^{-7}$. We generate 1000 random initial points within the hypercube $\text{int}(\mathcal{X})=\{x \in\R_+^n: 0 < x(i) < B\}$. The success rates of the classical PSG and the subgradient gliding method under different parameter settings are reported in Table \ref{tab:success_rate3}.

\begin{table}[htbp]
	\centering
	\caption{Success rate comparison between PSG and subgradient gliding method (SGM) in strongly convex setting}
	\begin{adjustbox}{width=0.8\textwidth}
		\begin{tabular}{ccccccc}
			\toprule 
			\multicolumn{6}{c}{\textbf{Success rate on step size \eqref{size4}}} & \\
			\cmidrule{1-7} 
			\textbf{Case} & \textbf{Method} & \textbf{$n=1$} & \textbf{$n=10$} & $n=100$ & $n=1000$&\\
			\midrule 
			\multirow{2}{*}{\textbf{$B=2$}} & \textbf{PSG}& $6.4\%$& $0\%$& $0\%$& $0\%$&\\
			\cmidrule{2-7}
			& \textbf{SGM} & $100\%$ & $100\%$ & $100\%$ & $100\%$ &\\
			\midrule 
			\multirow{2}{*}{\textbf{$B=1.5$}} & \textbf{PSG}& $36.2\%$& $0\%$& $0\%$& $0\%$&\\
			\cmidrule{2-7}
			& \textbf{SGM} & $100\%$ & $100\%$ & $100\%$ & $100\%$ &\\
			\midrule 
			\multirow{2}{*}{\textbf{$B=1.1$}} & \textbf{PSG}& $64\%$& $0.6\%$& $0\%$& $0\%$&\\
			\cmidrule{2-7}
			& \textbf{SGM} & $100\%$ & $100\%$ & $100\%$ & $100\%$ &\\
			\midrule 
			\multirow{2}{*}{\textbf{$B=1.01$}} & \textbf{PSG}& $85.9\%$& $24.6\%$& $0\%$& $0\%$&\\
			\cmidrule{2-7}
			& \textbf{SGM} & $100\%$ & $100\%$ & $100\%$ & $100\%$ &\\
			\midrule 
			\multirow{2}{*}{\textbf{$B=1.001$}} & \textbf{PSG}& $96.2\%$& $65.4\%$& $0.9\%$& $0\%$&\\
			\cmidrule{2-7}
			& \textbf{SGM} & $100\%$ & $100\%$ & $100\%$ & $100\%$ &\\
			\bottomrule 
		\end{tabular}
	\end{adjustbox}
	\label{tab:success_rate3}
\end{table}

As can be clearly seen from Table \ref{tab:success_rate3}, when $B > 1$, regardless of the specific value of $B$, the success rate of the classical PSG decays exponentially to zero as the problem dimension $n$ increases, whereas the subgradient gliding method achieves stable success while attaining the prescribed accuracy. These results provide strong empirical support for the argument in Theorem \ref{thm2.2}: even on simple problem instances with coordinate-wise separable objective functions, the classical PSG can fail with high probability.

We further demonstrate the flexibility of the subgradient gliding method in step size design. As established in Theorem \ref{thm4}, in the strongly convex setting, the subgradient gliding method achieves an $\mathcal{O}(1/t)$ convergence rate provided that the subgradient step size $\alpha_s$ and the gliding step size $\beta_s$ satisfy the prescribed conditions (C3) and (C2$^+$). This result not only provides a theoretical guarantee but also allows substantial flexibility in the practical choice of step sizes.

Figure \ref{fig8} illustrates the empirical convergence behavior of the algorithm under different choices of the gliding step size. The two panels correspond to two different sets of initial points. In each panel, the horizontal axis represents the iteration number, while the vertical axis shows, on a decibel scale, the gap between the function value and the optimum, namely $10 \log_{10}(f(x_s) - f^*)$. In the figure, “PSG” denotes the classical projected subgradient method with step size \eqref{size4}. “SGM ($\beta = 0.1/0.5/0.9$)” denotes the subgradient gliding method with a fixed gliding step size $\beta_s$ set to the corresponding constant. Finally, “SGM (ada\_beta)” denotes an adaptive variant in which, at each iteration, the gliding step size $\beta_s$ is selected from $\{0.1, 0.5, 0.9\}$ via a simple search (which automatically satisfies condition (C2$^+$)) and the subgradient step size $\alpha_s$ satisfies (C3) in Theorem \ref{thm4}, in order to achieve a larger decrease in the objective value.

\begin{figure}[htbp]
	\centering
	\includegraphics[width=\linewidth]{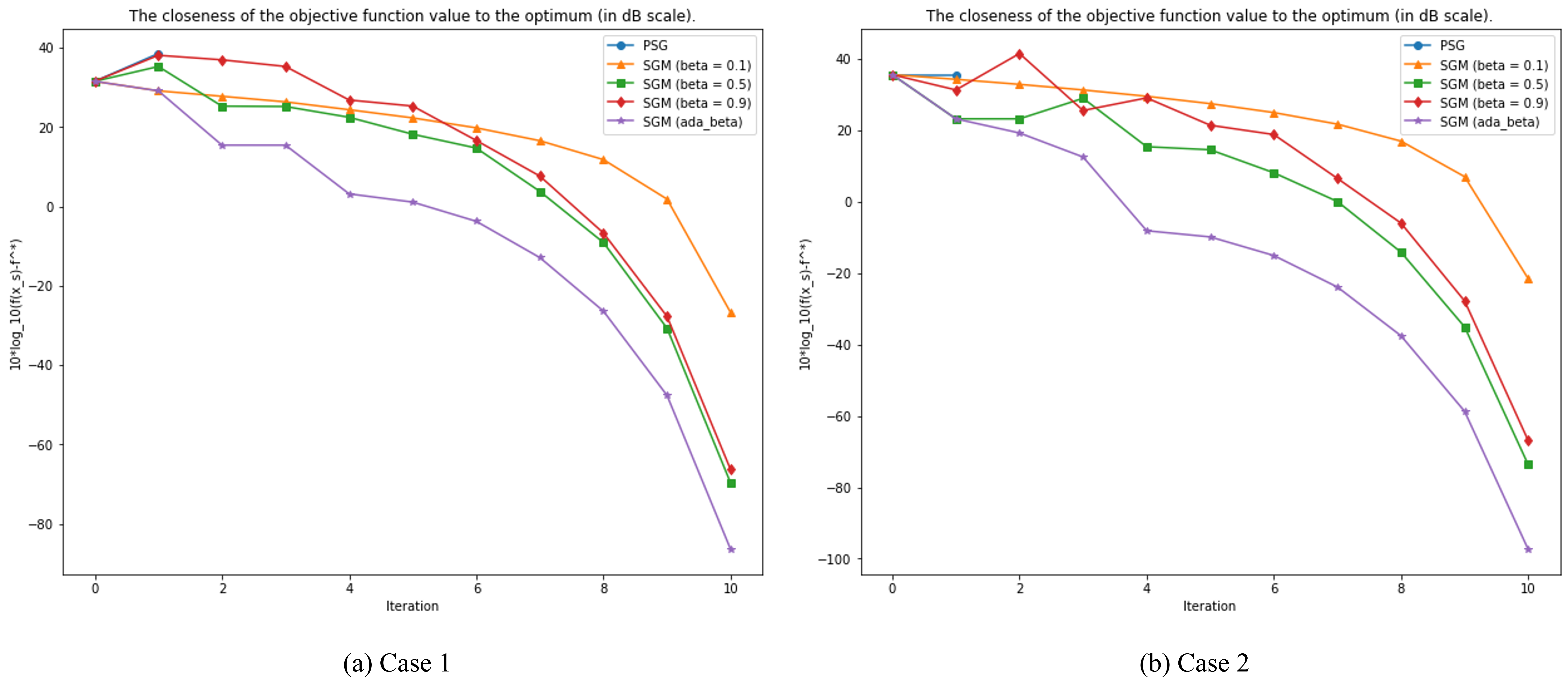}
	\vspace{-2mm}
	\caption{Convergence behavior of the subgradient gliding method under different choices of the gliding step size (the experimental setup here is $n=10000$, $B=2$).}
	\label{fig8}
\end{figure}

As can be seen from Figure \ref{fig8}, the subgradient gliding method converges stably within 10 iterations regardless of the choice of the gliding step size. Nevertheless, the specific choice of the gliding step size $\beta_s$ has a pronounced effect on the practical convergence speed. In particular, the subgradient gliding method with an adaptive gliding step size achieves an accuracy improvement of approximately 10–20 dB (i.e., one to two orders of magnitude) compared to its counterparts with a fixed gliding step size.

Although the adaptive step size considered here is obtained via a simple search over $\beta_s \in \{0.1, 0.5, 0.9\}$, the algorithm is still able to yield substantial gains by adaptively selecting small, moderate, or large gliding step sizes according to the position of current iterate. These observations further illustrate that, while preserving theoretical convergence guarantees, the subgradient gliding method offers considerable flexibility in the choice and design of step sizes in practical applications.

\section{Discussions on SGM's Convegence in Non-Lipschitz Convex Example \ref{e2} }
In this section, we analyze the convergence of the subgradient gliding method on Example \ref{e2}. As noted above, Example \ref{e2} is a convex function that fails to satisfy the Lipschitz condition on any level set, which poses significant challenges for the convergence analysis of subgradient methods. In Theorem \ref{cor:joint_design} and Remark \ref{r5}, we have shown that the convergence of the subgradient gliding method does not rely on global Lipschitz continuity, but instead depends on the behavior of the subgradient norms along the iterate sequence. This property provides a powerful analytical tool for applying the subgradient gliding method to non-Lipschitz convex optimization problems. We will demonstrate that, when employing the subgradient step-size family \eqref{size3} and gliding step size satisfying (C2$^+$), the subgradient gliding method can approach the optimal solution of Example \ref{e2} with arbitrary accuracy as the number of iterations increases.

As a preliminary step, we first illustrate the geometric properties of the subgradients in Example \ref{e2}. We establish the following lemma.

\begin{lemma}\label{geometry}
	Let $\angle(x,y) \in [0,\pi]$ denote the angle between vectors $x$ and $y$. Then, for Example \ref{e2}, the subgradient satisfies
	$$
	\angle(x_s, e_1) = \angle(-g_s, -x_s),
	$$
	where $e_1 = \left(1, 0\right)^T$. And, whenever $x_s(2) \neq 0$, the direction $-g_s$ is a descent direction of $\|g(x)\|$ at $x_s$, satisfying $\angle(-g_s, \nabla \|g(x_s)\|) = \frac{\pi}{2} + \angle(x_s, e_1)$.
\end{lemma}
\begin{proof}
	In Example \ref{e2}, the subgradient at points within the region $0 < x(1) \le 1,\; -1 \le x(2) \le 1$ satisfies
	\begin{equation*}
		g(x) = \left(1 - \frac{x(2)^2}{x(1)^2}, \frac{2\cdot x(2)}{x(1)}\right)^T.
	\end{equation*}
	and $\|g(x)\| = 1 + \frac{x(2)^2}{x(1)^2}$.
	Note that 
	\begin{equation*}
		\cos \angle(x_s, e_1) = \frac{x_s(1)}{\sqrt{x_s(1)^2 + x_s(2)^2}}
	\end{equation*}
	and 
	\begin{equation*}
		\cos \angle(-g_s, -x_s) = \frac{g_s^T x_s}{\|g_s\|\|x_s\|} = \frac{x_s(1)\left(1-\frac{x_s(2)^2}{x_s(1)^2}\right)+x_s(2)\frac{2x_s(2)}{x_s(1)}}{\sqrt{x_s(1)^2 + x_s(2)^2}\left(1 + \frac{x(2)^2}{x(1)^2}\right)} = \frac{x_s(1)}{\sqrt{x_s(1)^2 + x_s(2)^2}},
	\end{equation*}
	we have $\angle(x_s, e_1) = \angle(-g_s, -x_s)$.
	Since
	\begin{equation*}
		\nabla \|g(x)\| = \left( - \frac{2\cdot x(2)^2}{x(1)^3}, \frac{2\cdot x(2)}{x(1)^2}\right)^T,
	\end{equation*} 
	we have
	\begin{eqnarray}
		g_s^T\nabla \|g(x_s)\| &=& - \frac{2\cdot x(2)^2}{x(1)^3}\left(1 - \frac{x(2)^2}{x(1)^2}\right) + \frac{4\cdot x(2)^2}{x(1)^3} \nonumber \\
		&=& - \frac{2\cdot x(2)^2}{x(1)^3}\left(1 + \frac{x(2)^2}{x(1)^2}\right) \nonumber \\
		&>& 0 \nonumber
	\end{eqnarray} 
	when $x_s(2) \neq 0$. Hence, the direction $-g_s$ is a descent direction of $\|g(x)\|$ at $x_s$, i.e., $\angle(-g_s, \nabla \|g(x_s)\|) > \frac{\pi}{2}$. Specifically, note that 
	\begin{equation*}
		x_s^T\nabla \|g(x_s)\| =  - \frac{2\cdot x(2)^2}{x(1)^2} + \frac{2\cdot x(2)^2}{x(1)^2} = 0,
	\end{equation*} 
	$\angle(-x_s, \nabla \|g(x_s)\|) = \frac{\pi}{2}$. Hence, 
	\begin{equation*}
		\angle(-g_s, \nabla \|g(x_s)\|) = \angle(-g_s, -x_s) + \angle(-x_s, \nabla \|g(x_s)\|) = \frac{\pi}{2} + \angle(x_s, e_1),
	\end{equation*}
	as shown in Figure \ref{fig9}. The proof is complete.
	\begin{figure}[htbp]
		\centering
		\includegraphics[width=0.6\linewidth]{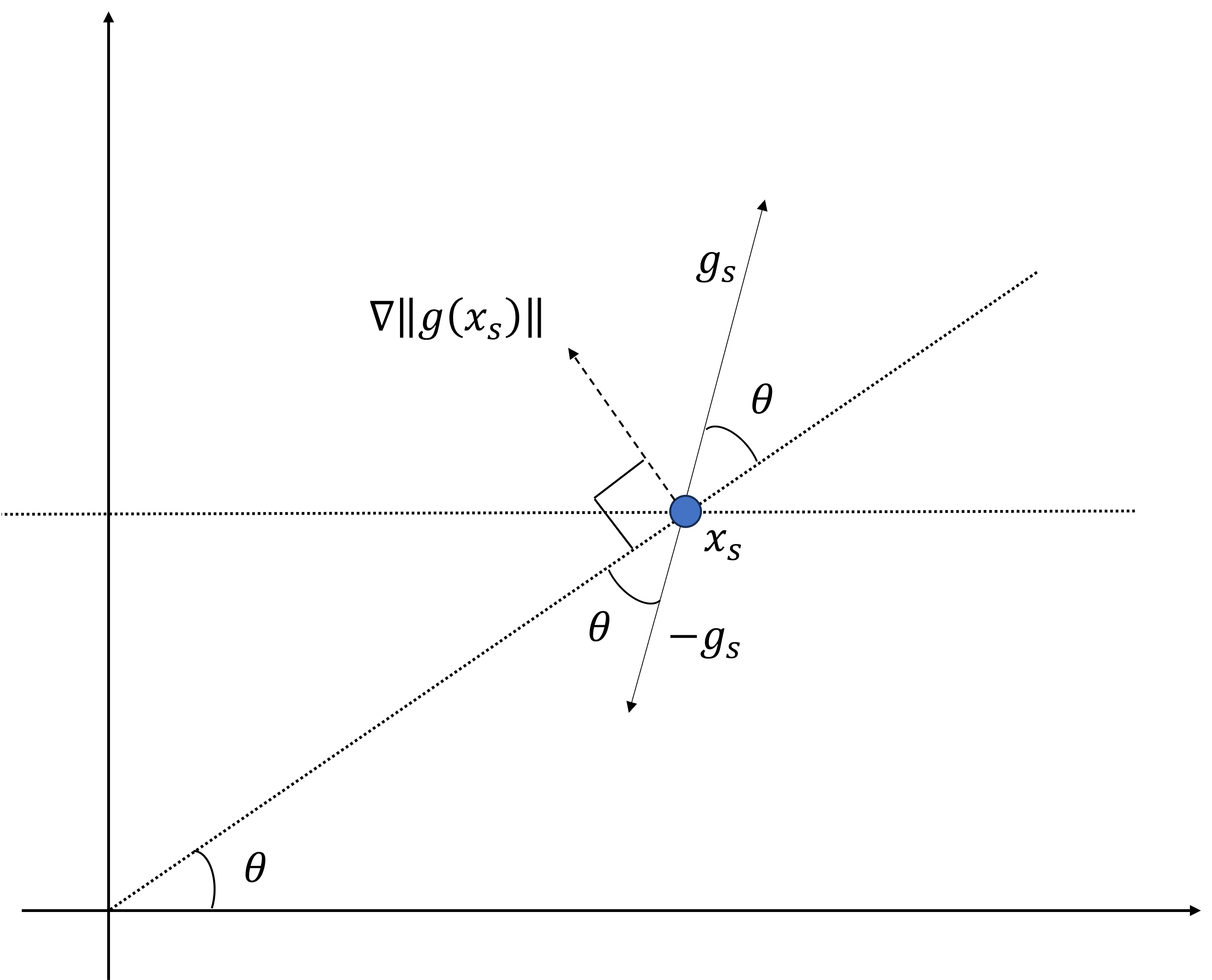}
		\caption{Geometry of subgradient in Example \ref{e2}.}
		\label{fig9}
	\end{figure}
\end{proof}

Based on the geometric properties established in Lemma \ref{geometry} together with Theorem \ref{cor:joint_design}, we present the following convergence result for the subgradient gliding method applied to Example \ref{e2}.

\begin{theorem}\label{convergence_analysis}
	Consider the optimization problem in Example \ref{e2} and apply the subgradient gliding method with the subgradient step-size family \eqref{size3} and gliding step size satisfying (C2$^+$). If the initial point $x_1 \in \operatorname{int}(\mathcal{X})$, then all iterates $x_s \in \operatorname{int}(\mathcal{X})$, and
	\begin{equation*}
		\liminf_{s \to \infty} \|x_s - x^*\| = 0.
	\end{equation*}
	That is, as the number of iterations increases, the iterates can approach the optimal solution arbitrarily closely.
\end{theorem}

\begin{proof}
	We note that, since the optimal solution of Example \ref{e2} is $x^* = (0,0)^T$, which lies on the boundary of the feasible set, the subgradient gliding method with gliding step sizes satisfying (C2$^+$) generates iterates satisfying $\|x_s - x^*\| > 0$ for all $s$. Therefore, the above theorem is in fact equivalent to proving that, for any $\epsilon > 0$, there exists an index $s$ such that $\|x_s - x^*\| \le \epsilon$. We divide the proof into two cases.
	
	1. There exists an iterate $x_s$ such that $x_s(2) = 0$. In this case, we have $g_s = (1,0)^T$. Consequently, all subsequent iterates remain on the $x(1)$-axis. This situation is straightforward, as it reduces to applying the subgradient gliding method to the one-dimensional function $f(x) = x$ on $[0,1]$. Since the Lipschitz condition holds in this setting, the method necessarily converges, and the desired conclusion follows.
	
	2. For all iterates $x_s$, we have $x_s(2) \neq 0$. In this case, we proceed by contradiction. Suppose that the method does not converge; that is, there exists some $\epsilon > 0$ such that $\|x_s\| > \epsilon$ for all $s$. Then, for sufficiently large $s$, we establish the following claim.
	
%
%
	
	\textbf{Claim:} For sufficiently large $s$, the subgradient norms of the iterates generated by the subgradient gliding method remain bounded.
	
	Observe that the norm of the subgradient at an interior point is directly related to the absolute value of the tangent at that point ($\|g(x)\| = 1 + \frac{x(2)^2}{x(1)^2}$): the larger the absolute tangent value, the larger the subgradient norm. Therefore, as illustrated in Figure \ref{fig10}, it suffices to show that, for sufficiently large $s$, starting from some iterate $x_s$, all subsequent iterates remain within the region spanned by the rays $OP$ and $OQ$, where $OP$ is aligned with $x_s$ and $OQ$ is the reflection of $OP$ with respect to the $x(1)$-axis.
	\begin{figure}[htbp]
		\centering
		\includegraphics[width=1.1\linewidth]{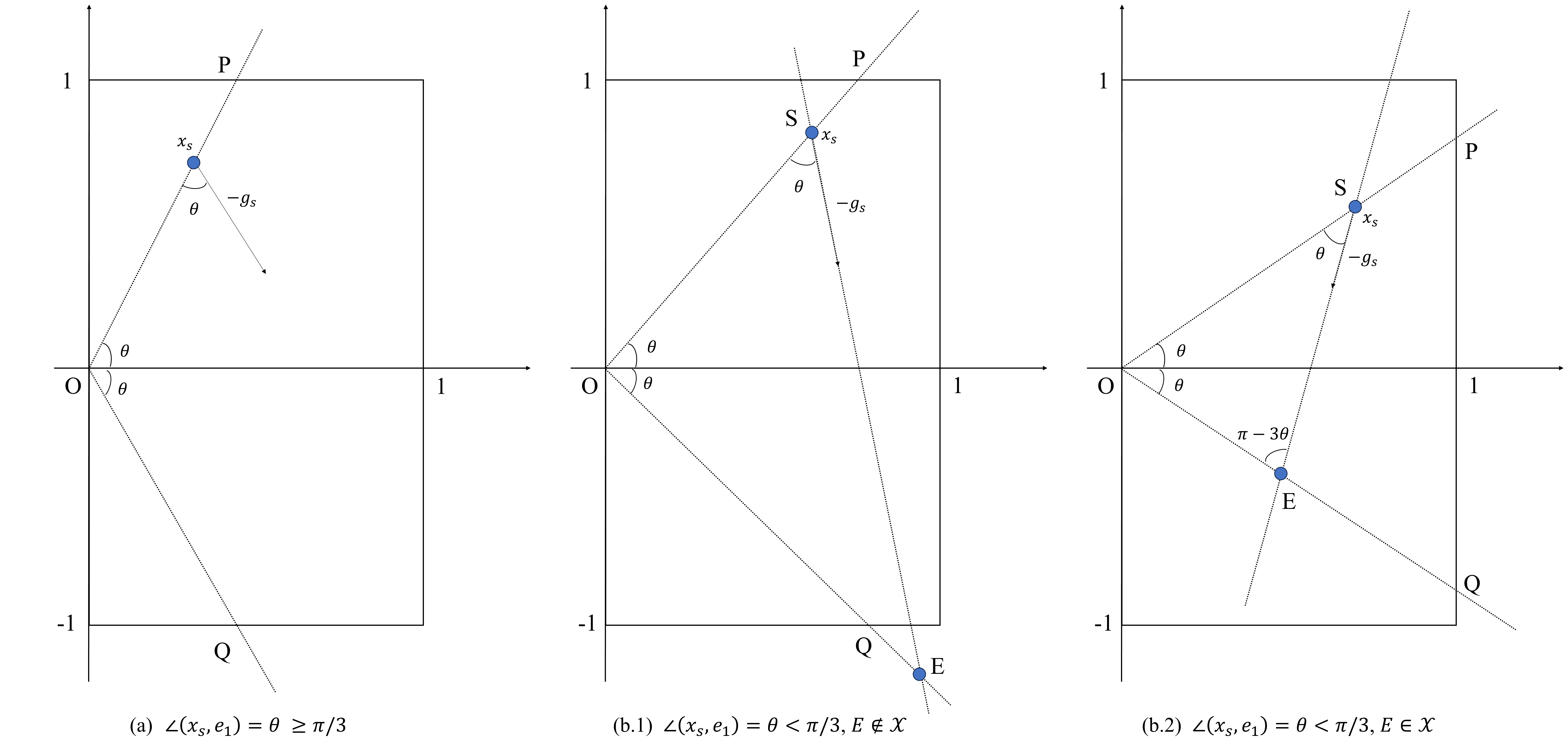}
		\caption{Subgradient norm remain bounded for sufficiently large $s$ under contradiction assumption.}
		\label{fig10}
	\end{figure}
	
	We prove the claim by considering two cases.
	
	(a) $\angle(x_s, e_1) = \theta \ge \pi/3$.
	This case is straightforward. By Lemma \ref{geometry}, the update along the direction $-g_s$ (even if a projection occurs afterward) produces the next iterate $x_{s+1}$ that remains within the region spanned by $OP$ and $OQ$. Consequently, $\|g(x_{s+1})\| \le \|g(x_s)\|,$ as illustrated in Figure \ref{fig10} (a).
	
	(b) $\angle(x_s, e_1) = \theta < \pi/3$. For this case, we denote the position of current iterate $x_s$ by $S$, and let $E$ be the intersection point between the ray starting from $S$ along the direction $-g_s$ and the ray $OQ$. We further divide the analysis into two subcases.
	
	(b.1) The intersection point $E$ lies outside the feasible region $\mathcal{X}$.
	This situation can only occur when $\pi/4 < \theta < \pi/3$. The analysis is also straightforward. When $x_s$ is updated along the direction $-g_s$, regardless of whether a projection occurs afterward, the new iterate $x_{s+1}$ remains within the region spanned by $OP$ and $OQ$. Consequently, $\|g(x_{s+1})\| \le \|g(x_s)\|,$
	as illustrated in Figure \ref{fig10} (b.1).
	
	(b.2) The intersection point $E$ lies inside the feasible region $\mathcal{X}$.
	This is the only case in which $\|g(x_{s+1})\|$ may potentially increase. We will show that, for sufficiently large $s$, the subgradient norm at $x_{s+1}$ still does not increase. In fact, when $s$ is sufficiently large, if $\|\alpha_s g_s\| \le SE$, then the subgradient gliding method ensures that the next iterate does not increase the subgradient norm and no projection occurs, as shown in Figure \ref{fig10} (b.2).
	
	We first show that, under the contradiction assumption, the segment length $SE$ admits a uniform positive lower bound. By the law of sines, we have
	\begin{equation*}
		\frac{OS}{\sin(\pi-3\theta)} = \frac{SE}{\sin(2\theta)},
	\end{equation*}
	which implies
	\begin{equation*}
		SE = \frac{\sin(2\theta)}{\sin(3\theta)}\cdot OS = \frac{\sin(2\theta)}{\sin(3\theta)}\cdot \|x_s\| > \frac{\sin(2\theta)}{\sin(3\theta)}\cdot \epsilon.
	\end{equation*}
	We then show that for $\theta \in (0, \pi/3)$, 
	\begin{equation*}
		\frac{\sin(2\theta)}{\sin(3\theta)} = \frac{2\cos\theta}{4\cos^2\theta - 1} > \frac{2}{3},
	\end{equation*}
	which holds since $\frac{2x}{4x^2-1}$ is decreasing on $(\frac{1}{2}, 1)$.	Therefore, the inequality $ SE > \frac{2}{3}\epsilon $ holds uniformly, independent of the value of $\theta$. For the subgradient step-size family \eqref{size3}, we have
		\begin{equation*}
				\|\alpha_s g_s\|
				\le
				\frac{R}{\max_{j=1,\cdots,s}\|g_j\|\sqrt{s}} \cdot \|g_s\|
				\le
				\frac{R}{\sqrt{s}}.
			\end{equation*}
	Therefore, when $s$ is sufficiently large ($s > \frac{9R^2}{4\epsilon^2}$, independent of the value of $\theta$), $\|\alpha_s g_s\| \le \frac{2}{3}\epsilon < SE$, and the subgradient norm at $x_{s+1}$ does not increase.
	
	Observe that the arguments in cases (a) and (b) above rely only on the contradiction assumption $\|x_s\| > \epsilon$ for all $s$ and the condition $s > \frac{9R^2}{4\epsilon^2}$, and are independent of the specific value of $\theta$. Since $\|x_{s+1}\| > \epsilon$ also holds under the contradiction assumption and $s+1 > \frac{9R^2}{4\epsilon^2}$, it follows by induction that the subgradient norms of all subsequent iterates do not increase and therefore remain bounded. Consequently, the sequence ${\|g_s\|}$ generated by the subgradient gliding method is bounded.
	
	By Theorem \ref{cor:joint_design} (4), the subgradient gliding method converges at the rate $\mathcal{O}(1/\sqrt{t})$ when $\{\|g_s\|\}$ is bounded. Hence, when the number of iterations $t$ is sufficiently large, there must exist an iterate $x_s$ such that 
	\begin{equation*}
		f(x_s) = \min_{j=1,\ldots,t} f(x_j) = \frac{x_s(1)^2 + x_s(2)^2}{x_s(1)} \le \epsilon^2,
	\end{equation*}
	which implies
	\begin{equation*}
		x_s(1)^2 + x_s(2)^2 \le \frac{x_s(1)^2 + x_s(2)^2}{x_s(1)} \le \epsilon^2,
	\end{equation*}
	i.e., $\|x_s\| \le \epsilon$. This contradicts the assumption. The proof is complete.
\end{proof}

\begin{remark}
	The derivation of Theorem \ref{convergence_analysis} provides a theoretical explanation for the convergence behavior observed in Section 6.2 (Figures \ref{fig4} and \ref{fig5}) for Example \ref{e2}. Suppose that the current accuracy satisfies $\|x_s - x^*\| > \epsilon$ and that the step length $\|\alpha_s g_s\|$ is much smaller than $\epsilon$ when $s$ is sufficiently large. In this regime, SGM does not project onto the left boundary $x(1)=0$; instead, as shown in the proof of Theorem \ref{convergence_analysis}, the iterates gradually approach the optimal solution $x^*$. Throughout this phase, the subgradient norms remain bounded.
	As the accuracy $\|x_s - x^*\|$ becomes comparable to the step length $\|\alpha_s g_s\|$, SGM may generate a projection onto the left boundary $x(1)=0$. This event may cause a temporary surge in the subgradient norm at a certain iteration (see Figure \ref{fig5}). For the subgradient step-size family \eqref{size3}, such a surge leads to a substantial reduction in the subsequent step lengths $\|\alpha_s g_s\|$, since $\alpha_s$ encodes the information $\max_{j \le s} \|g_j\|$ as $s$ increases. Consequently, the distance $\|x_s - x^*\|$ once again becomes significantly larger than the step length, and SGM enters a new phase of gradual approximation until $\|x_s - x^*\|$ again reaches the same order as the step length.
	
	By repeating this mechanism, the iterates progressively refine their accuracy. As the number of iterations increases, the algorithm approaches the optimal solution $x^*$ with arbitrary precision.
\end{remark}

As shown above, the subgradient gliding method not only effectively resolves the critical issue of premature termination caused by the absence of boundary subgradients, but also provides a fundamentally new analytical perspective for studying convergence in non-Lipschitz nonsmooth convex settings. In particular, global Lipschitz continuity is no longer a necessary assumption for the convergence of subgradient-type methods, relying only on the behavior of the subgradient norms along the iterate sequence.

\textbf{Open question:} An intriguing open question is whether global Lipschitz continuity can be entirely replaced by a subgradient growth paradigm in general nonsmooth convex optimization theory. More broadly, it remains open whether a unified theory can be developed to classify nonsmooth convex functions according to subgradient growth behavior rather than Lipschitz continuity, thereby delineating the maximal regime in which subgradient-type methods remain provably convergent.

\section{Conclusions}
This paper addressed a fundamental limitation of the classical projected subgradient  method (PSG) in nonsmooth convex optimization, i.e., its inevitable failure when valid subgradients are absent on the boundary of the feasible set. We have shown that even on simple problem instances and under standard step size, classical PSG can terminate after a single iteration with probability arbitrarily close to one, a failure mode that existing convergence analyses do not capture.

To overcome this intrinsic drawback, we introduced the subgradient gliding method (SGM), a novel alternative for both deterministic and stochastic settings that keeps all iterates strictly interior and thus remains well‑defined without any addtional assumptions on the existence of boundary subgradients in literatures. The proposed method enjoys strong theoretical guarantees: we established optimal ergodic convergence rates $\mathcal{O}(1/\sqrt{t})$ for convex objectives and $\mathcal{O}(1/t)$ for strongly convex objectives without requiring the restrictive global Lipschitz assumption. By relying instead on mild subgradient growth conditions, our framework broadens the scope of nonsmooth convex optimization problems to non‑Lipschitz regimes for which subgradient-based methods can be rigorously analyzed. 

Beyond its theoretical soundness, SGM offers significant practical advantages. Its admissible step size rules strictly contain those of PSG and provide greater flexibility in algorithm design. In numerical experiments where classical PSG fails completely, SGM achieves stable convergence with a $100\%$ success rate and exhibits orders‑of‑magnitude improvements in both accuracy and convergence speed. These results demonstrate that SGM substantially enlarges the design space of subgradient‑based methods. Future work includes extending recent advances on last-iterate convergence of PSG \cite{zamani2025exact} to subgradient gliding framework, accelerating the subgradient gliding method under appropriate problem parameters \cite{johnstone2020faster}, as well as investigating subgradient gliding methods and related theories in smooth and non‑convex settings.

\bmhead{Funding}

This work was supported by National Natural Science Foundation of China (Grant No. 125B2016, 12171021), and the Academic Excellence Foundation of BUAA for PhD Students.

\bmhead{Data Availability}
The manuscript has no associated data.

\bmhead{Conflict of interest}
The authors declare no competing interests.


\begin{thebibliography}{9}
	\bibitem{bubeck2015convex}
	Bubeck, S.: Convex optimization: algorithms and complexity. Foundations
	and Trends Trends{\textregistered} in Machine Learning, 8(3-4): 231-357 (2015).
	\bibitem{beck2017first}
	Beck, A.: First-order methods in optimization. Society for Industrial and Applied Mathematics (2017).
	\bibitem{nesterov2004lectures}
	Nesterov, Y.: Introductory lectures on convex optimization: A basic course. Springer Science \& Business Media (2004).
	\bibitem{nesterov2018lectures}
	Nesterov, Y.: Lectures on convex optimization, vol. 137. Springer, Berlin (2018).
	\bibitem{Zhu}
	Zhu, Z., Zhang, Y., Xia, Y.: Convergence rate of projected subgradient method
	with time-varying step-sizes. Optimization Letters, 19(5): 1027-1031 (2025). https://doi.org/10.1007/s11590-024-02142-9
	\bibitem{Xia}
	Xia, Y., Zhang, Y., Zhu, Z.: Lipschitz-free projected subgradient method
	with time-varying step-size. Journal of the Operations Research Society of China (2025). https://doi.org/10.1007/s40305-025-00629-5
	\bibitem{Lacoste-Julien}
	Lacoste-Julien, S., Schmidt, M., Bach, F.: A simpler approach to obtaining an $\mathcal{O}(1/t)$ convergence rate for the projected stochastic subgradient method. arXiv preprint arXiv:1212.2002 (2012).
	\bibitem{Rockafellar}
	Rockafellar, R. T.: Convex analysis. Princeton Math. Series (1970).
	\bibitem{Renegar}
	Renegar, J.: ``Efficient" subgradient methods for general convex optimization. SIAM Journal on Optimization, 26(4): 2649-2676 (2016).
	\bibitem{zamani2025exact}
	Zamani, M., Glineur, F.: Exact convergence rate of the last iterate in subgradient methods. 
	SIAM Journal on Optimization, 35(3): 2182--2201 (2025).
	\bibitem{alber1998projected}
	Alber, Ya.~I., Iusem, A.~N., Solodov, M.~V.: On the projected subgradient method for nonsmooth convex optimization in a Hilbert space. 
	Mathematical Programming, 81(1): 23--35 (1998).
	\bibitem{johnstone2020faster}
	Johnstone, P.~R., Moulin, P.: Faster subgradient methods for functions with H{\"o}lderian growth. 
	Mathematical Programming, 180(1): 417--450 (2020).
\end{thebibliography}
\end{document}